\documentclass[a4paper,12pt,leqno]{article}
\usepackage{amsmath}
\usepackage{amsthm}
\usepackage{amssymb}
\usepackage{amscd}
\usepackage{graphicx, color}
\usepackage[all]{xy}
\newtheorem{theorem}{Theorem}[section]
\newtheorem{lemma}[theorem]{Lemma}
\newtheorem{corollary}[theorem]{Corollary}

\theoremstyle{definition}
\newtheorem{definition}[theorem]{Definition}

\newtheorem{remark}[theorem]{Remark}

\numberwithin{equation}{section}
\usepackage[all]{xy}
\newcommand{\ep}{\epsilon}
\newcommand{\Hol}{\mbox{{\rm Hol}}}

\newcommand{\Z}{\Bbb Z}
\newcommand{\C}{\Bbb C}
\newcommand{\R}{\Bbb R}
\newcommand{\Q}{\Bbb Q}

\newcommand{\T}{\Bbb T}
\newcommand{\F}{\Bbb F}
\renewcommand{\P}{{\rm P}}
\newcommand{\SP}{\mbox{{\rm SP}}}

\newcommand{\RP}{\Bbb R\mbox{{\rm P}}}

\newcommand{\Map}{\mbox{{\rm Map}}}

\newcommand{\CP}{\Bbb C {\rm P}}

\newcommand{\dis}{\displaystyle}
\newcommand{\p}{\prime}

\newcommand{\Po}{\mbox{{\rm Poly}}}
\newcommand{\po}{\mbox{{\rm Pol}}}
\newcommand{\SZ}{{\mathcal{X}}^{d}}

\newcommand{\I}{\mbox{{\rm (i)}}}
\newcommand{\II}{\mbox{{\rm (ii)}}}
\newcommand{\III}{\mbox{{\rm (iii)}}}

\title{\bf The homotopy type of spaces of resultants
of bounded multiplicity}

\author{Andrzej Kozlowski\footnote{%
Institute of Applied Mathematics and Mechanics,
University of Warsaw, Banacha 2, 02-097 Warsaw, Poland
(E-mail: akoz@mimuw.edu.pl)
}
\  and \ 
Kohhei Yamaguchi\footnote{%
Department of Mathematics,
University of Electro-Communications,  Chofu, Tokyo 182-8585, Japan
(E-mail: kohhe@im.uec.ac.jp)
}}

\date{}

\begin{document}
\maketitle

%%%(Abstract)%%%%%.
\begin{abstract}
%%%%%%%%%%%
For positive integers $m,n, d\geq 1$ with $(m,n)\not= (1,1)$ and a field
$\F$ with its algebraic closure $\overline{\F}$, 
let $\Po^{d,m}_n(\F)$ denote the space  of all
$m$-tuples $(f_1(z),\cdots ,f_m(z))\in \F [z]$
of monic polynomials of the same degree $d$ such that
polynomials $f_1(z),\cdots ,f_m(z)$ have no common root 
in $\overline{\F}$ of multiplicity
$\geq n$. These spaces were defined  by Farb and Wolfson in \cite{FW} as generalizations of spaces first studied by Arnold, Vassiliev, Segal and others  in different contexts.
In \cite{FW} 
they obtained algebraic geometrical  and arithmetic results about the topology of these spaces. 
In this paper we  investigate the homotopy type of these spaces for the case 
$\F =\mathbb{C}$.  Our results generalize those of \cite{FW} for 
$\F =\C$ and also  results of 
G. Segal
\cite{Se}, V. Vassiliev \cite{Va} and F.Cohen-R.Cohen-B.Mann-R.Milgram \cite{CCMM} for $m\geq 2$ and $n\geq 2$.
\end{abstract}
%%%%%%%%%%%%%%%%%%%%%

%%%(SECTION 1)%%%%%%%%%%
\section{Introduction}\label{section 1}
%%%%%%%%%%%%%%%%%%%%%%%
%%%
\paragraph{1.1 Historical survey.}
%%%%%%%%%%%%%%%%%%%%%

There are two related intriguing phenomena that have been observed in various situations and can be roughly described as follows. 
Let  $X$ and $Y$ be two manifolds with some structure (e.g. holomorphic, symplectic, real algebraic) and let $\{M_d\}$ be a family of subspaces of structure-preserving continuous mappings $X\to Y$  indexed by  \lq\lq degree\rq\rq   (for a suitable notion of index)  %for some manifolds 
%and
with
\lq\lq stabilization mappings \rq\rq $q_d: M_d \to M_{d+1}.$ 
%which preserve the structure.   
The first phenomenon is that  the homology (homotopy) groups of the subspaces $M_d$ stabilize, that is: the mappings $q_d$ are homology (homotopy) equivalences up to some dimension $n(d)$ which is a monotonically increasing function of $d$. The other phenomenon is that a certain limit of the subspaces $M_d$ is homology (homotopy) equivalent to
the space  of all continuous mappings $X\to Y$. 

It seems that the first appearance of such phenomena was in the work of 
V. I. Arnold \cite{Ar}. 
Arnold considered the space $\SP_n^d(\C)$ of 
complex monic polynomials of the degree $d$ without roots of multiplicity $n$. 
For the case $n=2$ this is  the same as the space of 
monic polynomials without repeated roots (with non-zero discriminant), whose fundamental group is the braid group $\mbox{Br}(d)$ 
of $d$-strings and whose cohomology is the cohomology of the braid group
$\mbox{Br}(d)$. Arnold computed the homology of these braid groups and established their 
homological stability.  The corresponding relationship between these spaces of polynomials and spaces of continuous maps is given by the May-Segal theorem \cite{Se1}.   
Similar results also hold in the real polynomial case 
(\cite{GKY2}, \cite{KY1}, \cite{Va}). 
%%%
\par\vspace{2mm}\par
%%%%%%%%%%
Analogous another phenomena were discovered by G. Segal  \cite{Se} in a different context inspired by control theory (later it was discovered to have a close relationship with 
mathematical physics \cite{AJ}). 
Segal considered the space  
$\Hol^*_d(\CP^1,\CP^{n-1})$ 
of based holomorphic maps from $\CP^1$ to $\CP^{n-1}$ of 
the degree $d$ and its inclusion into the 
 space  $\Map_d^*(\CP^1,\CP^{n-1})=\Omega^2_d\CP^{n-1}$ 
of corresponding continuous maps.
Intuitive considerations based on Morse theory, suggest that homotopy of the first space should approximate that of the second space more and more closely as the degree $d$ increases. 
%%%%
Segal 
proved this result by observing that this space of based holomorphic mappings can be identified with the space  
of 
$n$-tuples  $(f_1(z),\dots,f_n(z))\in \C [z]^n$ 
of monic polynomials
of the same degree $d$ without common roots. 
%Segal
He 
defined a stabilization map 
%$\Rat _d(\CP^{n-1})\to \Rat _{d+1}(\CP^{n-1})$
$\Hol_d^*(\CP^1,\CP^{n-1})\to \Hol_{d+1}^*(\CP^1,\CP^{n-1})$ 
and  proved that the induced maps on homotopy groups are isomorphisms up to some dimension increasing with $d$.
%by a modification of Arnold's method. 
Using a different technique (based on ideas of Gromov and Dold-Thom) he also proved that 
there is a homotopy equivalence 
$q: \varinjlim \Hol_d^*(\CP^1,\CP^{n-1}) 
\stackrel{\simeq}{\longrightarrow} \Omega^2_0 \CP^{n-1}$ 
defined by a \lq\lq scanning of particles\rq\rq, 
 and that this equivalence is homotopic to the inclusion of the space of all holomorphic maps into the space of all continuous maps.
%That the relation between these was deeper than just the analogy between their proofs was noticed by Vassiliev \cite{Va}  who showed that the \lq\lq jet map\rq\rq (see below) induces a stable homotopy equivalence between the spaces $\Rat _d(\CP^2,\CP^n)$ and $\text{Poly}_n^{dn}$. 
\par 
With the help of the spectral sequence for complements of discriminants (analogous to the one he used in defining invariants of knots) Vassilev \cite{Va} 
 showed that 
%the 
%\lq\lq jet map\rq\rq (see below) induces 
there is a stable homotopy equivalence 
%between the spaces 
%$%\Rat _d(\CP^{n-1})
%%
%%(1.1)%%
\begin{equation}\label{eq: Va0}
%%%%%
\text{SP}_n^{d}(\C)
\simeq_s
\Hol_{\lfloor \frac{d}{n}\rfloor}^*(\CP^1,\CP^{n-1}),
\end{equation}
where
$\lfloor x\rfloor$ denotes the integer part of a real number $x$. 
\par
The relationship between Arnold's and Segal's arguments can  
also be explained in terms of Gromov's h-principle in \cite{GKY2} (cf. \cite{Gr}), and  in \cite{GKY4} 
it was shown %that (\ref{eq: Va0}) 
that there is
a homotopy equivalence\footnote{%
%%%(FootNote 1)%%%%%%
Since $\pi_1(\SP^d_2(\C))=\mbox{Br}(d)$, it is not commutative.
However, 
$\pi_1(\Hol^*_{\lfloor\frac{d}{n}\rfloor}(\CP^1,\CP^1))=\Z$.
So
two spaces 
$\SP^d_n(\C)$ and $\Hol^*_{\lfloor\frac{d}{n}\rfloor}(\CP^1,\CP^{n-1})$
are not homotopy equivalent if $n=2$.} 
%%(1.2)%%
\begin{equation}\label{eq: KY4-result}
%%%%%
\text{SP}_n^{d}(\C)
\simeq
\Hol_{\lfloor \frac{d}{n}\rfloor}^*(\CP^1,\CP^{n-1}),
\quad \mbox{ if }n\geq 3
%\ \ \mbox{ (see Theorem \ref{thm: GKY4}).}
%%%(End of FootNote 1)%%%
%%%%
\end{equation}
(see Theorem \ref{thm: GKY4}).
%between these spaces
%if $n\geq 3$. 
%%  
The argument makes use of the existence of a $C_2$-operad actions on the spaces $\coprod_{d\geq 0}\SP^d_n(\C)$
 and
$\coprod_{d\geq 0}\Hol_d^*(\CP^1,\CP^{n-1})$
(\cite{BM}, \cite{CS}, \cite{GKY4}).
%$\sqcup_{d\ge 0} \SP^d_n(\Bbb C)$ 
%and $\sqcup_{d\ge 0} \Hol _d(\CP^2,\CP^{n-1})$.  
%Although the two operad actions appear are defined in a rather analogous way, the jet map does not preserve operad action and there is no other obvious map that does. 
%%%%%%%%%
%%%%%%%%%
\par\vspace{2mm}\par
Recently Benson Farb and Jesse Wolfson \cite{FW} made a remarkable discovery.  
In order to state their results in full generality \footnote{%
%%(FootNote 2)%%%
They have been generalized  even farther in \cite{FWW}.
}
 Farb and Wolfson defined a new algebraic variety, given in terms of $m$-tuples of monic polynomials with conditions on common roots.  
Namely, for a field $\F$ with its algebraic closure $\bar{\F}$, 
let $ \Po^{d,m}_n(\F)$ denote the space of of $m$-tuples
$(f_1(z),\cdots ,f_m(z))\in \F [z]^m$ 
of monic polynomials of the same degree $d$ with no common root 
in $\overline{\F}$ of multiplicity $n$ or greater.
\par  
For example, if $\F =\C$,  $\Po^{d,1}_n (\C)=\SP^d_n(\C)$ %for $m=1$
%is the the space $\text{SP}_n^d(\C)$ of monic polynomials $f(z)\in \C[z]$
%of degree $d$ without roots of multiplicity $n$ (or greater), 
and 
$\Po^{d,n}_1 (\C)$ can be identified with the space
$\Hol_d^*(\CP^1,\CP^{n-1})$. %for $n=1.$
%$\Rat_d(\CP^{m-1})$ 
%of holomorphic maps from $\CP^1$ to $\CP^{n-1}$ of the degree $d$.
Note that 
in terms of this definition %the above mentioned 
the homotopy equivalence 
(\ref{eq: KY4-result}) can be  expressed as the homotopy equivalence
%%(1.3)%%%%%%%%%%%
\begin{equation}\label{equ: conj}
%%%%%%%%%%%%%%%
\Po^{d,1}_n (\C)
\simeq
\Po^{\lfloor\frac{d}{n}\rfloor,n}_1(\Bbb C) 
\quad
\mbox{ if }n\geq 3.
\end{equation}
%%%%%%%%%%%%%%%%%

By the classical theory of %discriminants and 
resultants, 
the space %$\Rat _d(\CP^{n-1})$
%$\Po^{d,n}_1(\C)=\Hol_d^*(\CP^1,\CP^{n-1})$ 
$\text{Poly}_n^{d,m}(\C)$ is an affine variety defined by systems of polynomial equations with integer coefficients.
Thus both varieties given in (\ref{equ: conj})
can be defined over $\Z$ and 
(by extension of scalars or reduction modulo a prime number) 
over any field $\F$.
%%%%
 Farb and Wolfson 
 computed various algebraic and arithmetic invariants 
 (such as the number of points for a finite field $\F_q$, etale cohomology etc) 
 of these varieties and found that these invariants are always equal.  They asked the natural question: are these varieties algebraically isomorphic for $n\ge 3$? 
\par
 In the simplest case $(d,n)=(3,3)$, %this was be shown to true by 
 Curt MacMullen constructed an isomorphism between
 $\text{Poly}_3^{3,1}(\Bbb Z[\frac{1}{3}])$ and
  $\Po^{1,3}_1(\Bbb Z[\frac{1}{3}])$ 
 and a different isomorphism between
$\text{Poly}_3^{3,1}(\Bbb Z/3)$ and 
  $\Po^{1,3}_1(\Bbb Z/3)$ 
  (\cite{FW}).  The formula defining the first of these isomorphisms of course gives also an isomorphism over $\C$ and $\R$, and hence, of course, implies that these spaces are homeomorphic and thus homotopy equivalent. 
 This example suggests that it is unlikely that the varieties  $\Po_n^{dn,1}(\Bbb Z)$ and
 $\Po_1^{d,n}(\Bbb Z)$ are isomorphic but are more likely to be so if we invert the primes dividing $d$. 
 %%%%%
As before, such an isomorphism of varieties (over a local ring)  induces an isomorphism over both $\Bbb C$ and $\Bbb R$. 
The question posed by Farb and Wolfson seems difficult to answer, and doing so would certainly require completely different methods from the ones used here. However, our results will lead to its generalization. Namely, in this paper we will show that %for  $mn\ge 3$ 
there is a homotopy equivalence

%%%%%%%%(1.4)%%%
\begin{equation}\label{equ: conj1}
%%%%%%%%%%%%%%%
\Po^{d,m}_n (\C) \simeq \Po^{\lfloor\frac{d}{n}\rfloor,mn}_1(\Bbb C)
\quad
\mbox{ if }mn\geq 3
\end{equation}
%%%%%%%%%%%%%%%%
(see Theorem \ref{thm: IV}).
 This naturally leads to the question of whether this homotopy equivalence derives from an isomorphism of varieties. Such an isomorphism should be defined for varieties over $\Q$ and perhaps over the ring $\Bbb Z[S^{-1}]$  where $S$ is some set of primes. 
Of course if this is the case,  the same must hold with coefficients in $\Bbb R$.  We should therefore expect to be able to prove homotopy equivalence in the real case too.
We intend to pursue this topic in the future work \cite{KY11}.
%%%
\par%\vspace{1mm}\par
We would like to note here the key role played in our argument  by the \lq\lq jet map\rq\rq\  
(see (\ref{equ: jet}) below).
%%% 
 The jet map is an algebraic map and can be defined over any field $\F$.
 Over $\C$ it induces a homotopy equivalence of suitable stabilizations 
 (Theorem \ref{thm: natural map})   but  it is not a candidate for the conjectured isomorphism since it does not have the right target space.  
%%%
\par\vspace{1mm}\par
%%%%%%
 It is interesting to observe that  
 there is another \lq\lq real analogue \rq\rq of  (\ref{equ: conj})
 given in
\cite{KY1}.  It  is obtained by having on the left hand side  the space of real polynomials without real roots of multiplicity $\ge n$ (considered by Arnold and Vassiliev \cite{Va}) and on the right hand side the space of real rational maps from $\RP^1$ to $\RP^{n-1}$ in the sense of Mostovoy  \cite{Mo1}, which can be identified with the space 
 represented by $n$-tuples of real coefficients monic polynomials of the same degree without common real roots (but possibly with common complex roots).  
Since these spaces are semi-algebraic varieties rather than algebraic varieties,
 the above argument  does not seem to apply,
but intriguingly the analogous result of (\ref{equ: conj1})
%this paper 
%that the spaces are homotopy equivalent 
remains true.  This 
will be proved in another work \cite{KY10}.
%%%
%%%
\par\vspace{2mm}\par

%\paragraph{1.2. The purpose of this paper.}
The purpose of this article is to determine the homotopy type of the  space $ \Po^{d,m}_n(\F)$ for $\F =\C$
which, from now on, will be denoted simply as 
$ \Po^{d,m}_n=\Po^{d,m}_n(\C)$ (see Definition \ref{dfn: Polydmn}).   The homotopy equivalence (\ref{equ: conj1}) is an immediate consequence of this  (Theorem \ref{thm: IV}).
Our arguments are generally analogous to those in \cite{GKY2}, \cite{GKY4} and \cite{KY6}, but the technical details are more complicated. 
%As already mentioned, we will consider real analogues of this paper in
%the subsequent papers \cite{KY10} and \cite{KY11}.

\paragraph{1.2 Basic definitions and notations.}

For connected spaces $X$ and $Y$, let
$\Map(X,Y)$ (resp. $\Map^*(X,Y)$) denote the space
consisting of all continuous maps
(resp. base-point preserving continuous maps) from $X$ to $Y$
with the compact-open topology.
When $X$ and $Y$ are complex manifolds, we denote by
$\Hol (X,Y)$ (resp. $\Hol^*(X,Y))$ the subspace of
$\Map (X,Y)$ (resp. $\Map^*(X,Y))$ consisting of all holomorphic maps
(resp. base-point preserving holomorphic maps).
\par
For each integer $d\geq 0$, let
$\Map_d^*(S^2,\CP^{N-1})=\Omega^2_d\CP^{N-1}$ denote the space of all
based continuous maps
$f:(S^2,\infty)\to (\CP^{N-1},*)$ such that
$[f]=d\in \Z=\pi_2(\CP^{N-1})$,
where we identify $\CP^1=S^2=\C \cup \{\infty\}$
and the points $\infty\in S^2$ and
$*\in \CP^{N-1}$ are the base points of $S^2$
and $\CP^{N-1}$, respectively.
%%%
Let $\Hol_d^*(S^2,\CP^{N-1})$ denote the subspace of 
$\Map_d^*(S^2,\CP^{N-1})$
consisting of all based holomorphic maps of degree $d$.
%%%%%%%%%%%%%%%%

%%(Remark 1.1)%%%
\begin{remark}
%%%%%%%%
Let $\P^d(\C)$ denote the space consisting of all complex
monic polynomials 
$f(z)=z^d+a_1z^{d-1}+\cdots +a_d\in \C [z]$ of the degree $d$.
%Let $\hat{\P}^d(\C)$ be the space of all polynomials
%$g(z)=\sum_{k=1}^da_kz^{d-k}\in \C [z]$ of degree $\leq d-1.$
%Clearly %$\P^d(\C)=\{z^d+x:x\in\hat{\P}^d(\C)\}$ and
%there is a homeomorphisms $\P^d(\C)\cong\hat{\P}^d(\C)\cong \C^d$.
%%
\par
If we choose the point $[1:1:\cdots :1]\in\CP^{N-1}$ as its base point,
the space $\Hol_d^*(S^2,\CP^{N-1})$ can be identified with the space consisting
of all $N$-tuples $(f_1(z),\cdots ,f_{N}(z))\in \P^d(\C)^{N}$
of monic polynomials of the same degree $d$ such that polynomials
$f_{1}(z),\cdots ,f_{N}(z)$ have no common root, i.e.
the space $\Hol_d^*(S^2,\CP^{N-1})$ can be identified with
%%%()%%
{\small
\begin{equation*}
\Hol_d^*(S^2,\CP^{N-1})=
\big\{(f_1,\cdots ,f_{N})\in \P^d(\C)^{N}:
\{f_k(z)\}_{k=1}^N
\mbox{ have no common root}\big\}.
\end{equation*}
}
%%%%
\end{remark}
%%%%%%%(End of Remark 1.1)%%%

%%(Definition 1.2)%%%
\begin{definition}\label{dfn: Polydmn}
%%(i)%%%%
(i)
Let $\SP^d_n(\C)$ denote the space of all monic polynomials
$f(z)\in \P^d(\C)$ of the degree $d$ without root of multiplicity
$\geq n$.
More generally,
for positive integers $m,n,d\geq 1$ with $(m,n)\not= (1,1)$,
let $\Po^{d,m}_n$ denote the space of all
$m$-tuples $(f_1(z),\cdots ,f_m(z))\in (\P^d(\C))^m$ of
monic polynomials of the same degree $d$
such that
polynomials $f_1(z),\cdots ,f_m(z)$ have no common root of multiplicity
$\geq n$.
%%%
\par
(ii)
Let $(f_1(z),\cdots ,f_m(z))\in \P^d(\C)^m$.
Note that
$(f_1(z),\cdots ,f_m(z))\in \Po^{d,m}_n$ iff
the polynomials
$\{f_j^{(k)}(z):1\leq j\leq m,\ 0\leq k<n\}$
have no common root.
%%
%%%
In this situation, define {\it the jet map}
%%%%%%%
%%%(1.5)%%
\begin{equation}\label{equ: jet}
%%%%%%%%
j^{d,m}_n:\Po^{d,m}_n\to\Hol_d^*(S^2,\CP^{mn-1})
\qquad
\mbox{by}
\end{equation}
%%%%
%%%(1.6)%%%%%%%%
\begin{equation}\label{equ: jet embedding}
%%%%%%%%%%%%%%%
j^{d,m}_n(f_1(z),\cdots ,f_m(z))=
(\textit{\textbf{f}}_1(z),\cdots ,\textit{\textbf{f}}_m(z))
\end{equation}
%%%%%%%%%%%%%%
for $(f_1(z),\cdots ,f_m(z))\in \Po^{d,m}_n$, 
where
$\textit{\textbf{f}}_k(z)$ $(k=1,2,\cdots ,m)$
denotes the $n$-tuple of monic polynomials of the same degree $d$ defined by
%%%(1.7)%%%
\begin{equation}\label{equ: bff}
%%%%%%%%
\textit{\textbf{f}}_k(z)=(f_k(z),f_k(z)+f^{\p}_k(z),f_k(z)+f^{\p\p}_k(z),
\cdots ,f_k(z)+f^{(n-1)}_k(z)).
%%%
\end{equation}
%%%%%%%%%%%
\par
Let $i^{d,m}$ denote the natural map
$i^{d,m}_n:\Po^{d,m}_n\to \Omega^2_d\CP^{mn-1}\simeq \Omega^2S^{2mn-1}$ defined by
%%%
%%%(1.8)%%%%%%%
\begin{equation}\label{equ: jet inclusion}
%%%%%%%%%%%%%%%
i^{d,m}_n(f_1(z),\cdots ,f_m(z))(\alpha)
=
\begin{cases}
[\textit{\textbf{f}}_1(\alpha):\cdots :\textit{\textbf{f}}_m(\alpha)]
& \mbox{ if }\alpha \in\C
\\
[1:1:\cdots :1]
%[\textit{\textbf{e}}_1:\textit{\textbf{e}}_1\cdots :\textit{\textbf{e}}_1]
& \mbox{ if }\alpha =\infty
\end{cases}
%%%%%%%%%%%%%%
\end{equation} 
%%%%%%%%%%%%%
for $(f_1(z),\cdots ,f_m(z))\in \Po^{d,m}_n$ and
$\alpha \in S^2=\C \cup \infty$.
\qed
%%%%%%%%%%%%%%%
\end{definition}
%%%(End of Definition 1.2)%%%%%

%%%%(Remark 1.3)%%%
\begin{remark}\label{rmk: 1.3}
%%%%%%%%%
(i)
Note that $\SP^d_n(\C)=\Po^{d,1}_n$ for $m=1$, and
that  
we can  identify $\Po^{d,m}_1=\Hol_d^*(S^2,\CP^{m-1})$
for $n=1$.
It is easy to see that there is a homeomorphism
%%(1.9)%%
\begin{equation}\label{eq: Poly}
%%%%%%%%
\Po^{d,m}_n
\cong
\begin{cases}
\C^{dm} & \mbox{if }d<n,
\\
\C^{mn}\setminus
\{(x,\cdots ,x):x\in\C\} & \mbox{if }d=n.
\end{cases}
\end{equation}
%%%
Thus in this paper we shall mainly consider the case $m\geq 2$ with 
$d\geq n\geq 2$.
%%%
\par
(ii)
A map $f:X\to Y$ is called {\it a homotopy equivalence} 
(resp. {\it a homology equivalence}) {\it through dimension} $D$
if the induced homomorphism
$f_*:\pi_k(X)\to \pi_k(Y)$
(resp. $f_*:H_k(X,\Z)\to H_k(Y,\Z))$
is an isomorphism for any $k\leq D$.
\qed
\end{remark}
%%%(Endo of Remark 1.3)%%

%\par\vspace{1mm}\par

%%%
\paragraph{1.3 The related known results. }
%%%%%%%%%%%
Now, recall the following known results.
First, consider the case $m=1$.
Note that $\Po^{d,1}_n=\SP^d_n(\C)$.

%%%%(Theorem 1.4: GKY2, KY7)%%%
\begin{theorem}
[\cite{GKY2}, \cite{KY7}]
\label{thm: KY7}
%%%%%%%%%%%%%%%%%%%%%%%%%%%%%%%
The jet map 
$$
j^{d,1}_n:\SP^d_n(\C)\to \Omega^2_d\CP^{n-1}
$$ is a
homotopy equivalence through dimension
$(2n-3)(\lfloor \frac{d}{n}\rfloor +1)-1$ if $n\geq 3$ and 
it is a
homology equivalence thorough dimension $\lfloor \frac{d}{2}\rfloor$
if $n=2$. 
%where
%$\lfloor x\rfloor$ denotes the integer part of a real number $x$.
\qed
\end{theorem}
%%%%%%%%%%%%%%%%%

Next, consider the case $m\geq 2$ and $n=1$.
In this case,  we can identify $\Po^{d,m}_1=\Hol_d^*(S^2,\CP^{m-1})$
and the following result is known.

%%(Theorem 1.5. : Segal+KY6 theorem)%%
\begin{theorem}[\cite{KY6}, \cite{Se}]\label{thm: KY6}
%%%%
If $m\geq 2$, the inclusion map
$$
i_d=j^{d,m}_1:\Hol_d^*(S^2,\CP^{m-1})\to \Omega^2_d\CP^{m-1}
\simeq \Omega^2S^{2m-1}
$$
is a homotopy equivalence through dimension 
$(2m -3)(d+1)-1$.
\qed
\end{theorem}
%%%%%%

We also recall the stable result obtained by
F.Cohen-R.Cohen-B.Mann-R.Milgram and its improvement due to
R.Cohen-D.Shimamoto
(\cite{CCMM}, \cite{CCMM2}, \cite{CS}).
%%
%%(Theorem 1.6: CCMM)%%
\begin{theorem}[\cite{CCMM}, \cite{CCMM2}, \cite{CS}]
\label{thm: CCMM}
%%%%%%%%%%%%%%%%
$\I$
If $m\geq 2$,
there is a stable homotopy equivalence
$$
\Hol_d^*(S^2,\CP^{m-1})\simeq_s
\bigvee_{k=1}^d\Sigma^{2(m-2)k}D_k,
$$
where $\Sigma^kX$ denotes the $k$-fold reduced suspension of a
based space $X$ and $D_k$ is the equivariant half smash product 
$D_k=F(\C,k)_+\wedge_{S_k}(\wedge^kS^1)$
defined by (\ref{equ: DkS}).
\par
$\II$ In particular, if $m\geq 3$, there is a homotopy equivalence
$$
\Hol_d^*(S^2,\CP^{m-1})\simeq
J_2(S^{2m-3})_d,
$$
where $J_2(S^{2m-3})_d$ denotes the $d$-th stage filtration of 
the May-Milgram model for $\Omega^2S^{2m-1}$ defined by 
(\ref{eq: d-th MM-model}).
\qed
\end{theorem}
%%%%%%%%%%%%%%

Note that the homotopy types of the spaces
$\SP^d_n(\C)$ and $\Hol^*_{\lfloor \frac{d}{n}\rfloor}(S^2,\CP^{n-1})$
are closely connected
in the following sense.

%%%(Theorem 1.7: GKY4)%%
\begin{theorem}[\cite{GKY4}, \cite{Va}]\label{thm: GKY4}
%%%%%%%%%%%%%%%%%
\begin{enumerate}
\item[$\I$]
If $n\geq 3$, there is a homotopy equivalence\footnote{%
%%(FootNote 2)%%%
This result was also proved independently by S. Kallel.
}
$$
\SP^d_n(\C)\simeq 
\Hol^*_{\lfloor \frac{d}{n}\rfloor}(S^2,\CP^{n-1}).
$$
\item[$\II$]
If $n=2$, there is a stable homotopy equivalence%\footnote{%
%%%%%%%%%%%%%%%%%%%%%%%%%
%%%()%%
%If $d\geq 2$, $\pi_1(\SP^d_2(\C))=\mbox{Br}(d)$ 
%%is isomorphic to the Braid group
%%$\mbox{Br}(d)$ of $d$-strings 
%and it is not commutative.
%On the contrary, since $\pi_1(\Hol_{\lfloor\frac{d}{2}\rfloor}(S^2,\CP^1))=\Z$ (\cite{Se}),
%two spaces $\SP^d_2(\C)$ and 
%$\Hol_{\lfloor\frac{d}{2}\rfloor}(S^2,\CP^1)$
%have the different homotopy types if $n=2$.
%}
%%(End of FootNote 1)%%%
%%%%
$$
\SP^d_2(\C)\simeq_s \Hol^*_{\lfloor \frac{d}{2}\rfloor}(S^2,\CP^{1}).
\qed
$$
\end{enumerate}
\end{theorem}
%%%%%%%%%%%%%%%

%%%%%%%%%%%%%%%%%%%%%%%%%%%%%%%
\paragraph{1.4 The main results. }
%%%%%%%%%%%%%%%%%%%%%%%%%%%%%%%
The main purpose of this paper is to investigate the homotopy type of the space $\Po^{d,m}_n$ and 
generalize the above results.
Since the case $m=1$ or the case $n=1$ was already well studied in the above theorems, 
we mainly consider the case $m\geq 2$ and $n\geq 2$.
Let $D(d;m,n)$ denote the positive integer defined by
%%%%%%
%%%%%%%%%%%%(1.10)%%%%%%%%%%%%%%%%%%
\begin{equation}\label{equ: number D}
%%%%%%%%%%%%%%%%%%%%%%%%%
D(d;m,n)=(2mn-3)(\Big\lfloor \frac{d}{n}\Big\rfloor +1)-1.
%%%%%%%%%%%%%%%%%%%%%%
\end{equation}
%%%%%% 
The main results of this paper is stated as follows.

%%(Theorem 1.8)%%%%
\begin{theorem}\label{thm: I}
%%%%%%%%%%%%%%%
Let $m$ and $n$ be positive integers.
If $mn\geq 3$,
the natural map
$$
i^{d,m}_n:\Po^{d,m}_n \to \Omega^2_d\CP^{mn-1}\simeq \Omega^2S^{2mn-1}
$$
is a homotopy equivalence through dimension
$D(d;m,n)$. 
%if $n\geq 3$ and  it is a homology equivalence through dimension
%$D(d;m,2)$ if $n=2$.
%%%%%%%%%%%%%
\end{theorem}
%%%%%%%%%%%%%
%%

%%(Corollary 1.9)%%%%
\begin{corollary}\label{thm: II}
%%%%%%%%%%%%%%%
Let $m$ and $n$ be positive integers.
If $mn\geq 3$,
the jet map
$$
j^{d,m}_n:\Po^{d,m}_n \to \Hol_d^*(S^2,\CP^{mn-1})
$$
is a homotopy equivalence through dimension
$D(d;m,n)$. 
%if $n\geq 3$ and it is a homology equivalence through dimension
%$D(d;m,2)$ if $n=2$.
%%%%%%%%%%%%%
\end{corollary}
%%%%%%%%%%%%%

By using the above result and the result obtained by
R. Cohen and D. Shimamoto \cite{CS}, we also obtain the following two results.
%%%%%%%(Theorem 1.10)%%
\begin{theorem}\label{thm: III}
%%%%%%%%%
%Under the same assumption as Theorem \ref{thm: III},
%If $n\geq 2$,
If $m$ and $n$ are positive integers with $(m,n)\not= (1,1)$,
there is a stable homotopy equivalence
$$\dis
\Po^{d,m}_n\simeq_s
\bigvee_{k=1}^{\lfloor \frac{d}{n}\rfloor}
\Sigma^{2(mn-2)k}D_k.
$$
%%%)%%%%%%
\end{theorem}
%%%(End of Theorem 1.10)%%%%%%%

%%%(Theorem 1.11)%%
\begin{theorem}\label{thm: IV}
%%%%%%%%%%%%%
Let $m$ and $n$ be positive integers.
If $mn\geq 3$, then
there is a homotopy equivalence
$$\dis
\Po^{d,m}_n\simeq 
\Hol^*_{\lfloor\frac{d}{n}\rfloor}(S^2,\CP^{mn-1}).
$$
%%%%
\end{theorem}
%%%%%%%%%%%%%%

%%(Remark 1.12)%%%%
\begin{remark}\label{rmk: Theorem IV}
%%%%%%%%%%%%%%%%%%
It is easy to see that the above homotopy equivalence given in Theorem \ref{thm: IV} can also be expressed 
in the form (\ref{equ: conj1}).
\qed
\end{remark}

%$$\Po^{d,m}_n\simeq \Po^{{\lfloor\frac{d}{n}\rfloor},n m}_1 $$

%%%%%%%%%%%%%%%%%%
%\par\vspace{2mm}\par
%%%%%%%
This paper is organized as follows.
In \S \ref{section: simplicial resolution} we recall the simplicial resolutions,
and
in \S \ref{section: spectral sequence} we shall construct the Vassiliev type spectral sequences converging to the homologies of
the spaces $\Po^{d,m}_n$ and $\Hol_d^*(S^2,\CP^{mn-1})$.
In \S \ref{section: sd}, we consider the stabilization map
$s^{d,m}_n:\Po^{d,m}_n\to \Po^{d+1,m}_n$ and
prove the homological stability of the map $s^{d,m}_n$
(Theorem \ref{thm: stab1}).
In \S \ref{section: scanning maps} we prove Theorem
\ref{thm: natural map} by using the scanning maps
and we give the proofs of the main 
 results of this
paper (Theorem \ref{thm: I} and Corollary \ref{thm: II}).
In \S \ref{section: III} we prove Theorem
\ref{thm: III} and Theorem \ref{thm: IV} by using Theorem \ref{thm: V}.
 In \S \ref{section: Proof V} we give the proof of Theorem
\ref{thm: V} by using Lemma \ref{lmm: X}, and
in \S \ref{section: transfer} we prove Lemma \ref{lmm: X}.

%%

%%(SECTION 2)%%%
\section{Simplicial resolutions}\label{section: simplicial resolution}
%%%%%%%%%%%%%%%%

In this section, we give the definitions of and summarize the basic facts about  non-degenerate simplicial resolutions  
(\cite{Va}, \cite{Va2}, (cf.  \cite{Mo2})).
%%%%%
%%(Definition 2.1)%%
\begin{definition}\label{def: def}
%%%%%%
{\rm
(i) For a finite set $\textbf{\textit{v}} =\{v_1,\cdots ,v_l\}\subset \R^N$,
let $\sigma (\textbf{\textit{v}})$ denote the convex hull spanned by 
$\textbf{\textit{v}}.$
%%%
%%%
Let $h:X\to Y$ be a surjective map such that
$h^{-1}(y)$ is a finite set for any $y\in Y$, and let
$i:X\to \R^N$ be an embedding.
Let  $\mathcal{X}^{\Delta}$  and $h^{\Delta}:{\mathcal{X}}^{\Delta}\to Y$ 
denote the space and the map
defined by
%%%
%%(2.1)%%%%%%%%
\begin{equation}
%%%%%%%%%%%%%%%
\mathcal{X}^{\Delta}=
\big\{(y,u)\in Y\times \R^N:
u\in \sigma (i(h^{-1}(y)))
\big\}\subset Y\times \R^N,
\ h^{\Delta}(y,u)=y.
\end{equation}
%%%%%%%%%%%%%%%%
The pair $(\mathcal{X}^{\Delta},h^{\Delta})$ is called
{\it the simplicial resolution of }$(h,i)$.
In particular, $(\mathcal{X}^{\Delta},h^{\Delta})$
is called {\it a non-degenerate simplicial resolution} if for each $y\in Y$
any $k$ points of $i(h^{-1}(y))$ span $(k-1)$-dimensional simplex of $\R^N$.
%%%%
\par
(ii)
For each $k\geq 0$, let $\mathcal{X}^{\Delta}_k\subset \mathcal{X}^{\Delta}$ be the subspace
given by 
%%(2.2)%%
\begin{equation}
%%%%%%%%
\mathcal{X}_k^{\Delta}=\big\{(y,u)\in \mathcal{X}^{\Delta}:
u \in\sigma (\textbf{\textit{v}}),
\textbf{\textit{v}}=\{v_1,\cdots ,v_l\}\subset i(h^{-1}(y)),\ l\leq k\big\}.
\end{equation}
%%%%
%%%%
We make identification $X=\mathcal{X}^{\Delta}_1$ by identifying 
 $x\in X$ with %the pair
$(h(x),i(x))\in \mathcal{X}^{\Delta}_1$,
and we note that  there is an increasing filtration
%%%(2.3)%%
\begin{equation*}\label{equ: filtration}
%%%
\emptyset =
\mathcal{X}^{\Delta}_0\subset X=\mathcal{X}^{\Delta}_1\subset \mathcal{X}^{\Delta}_2\subset
\cdots \subset \mathcal{X}^{\Delta}_k\subset \mathcal{X}^{\Delta}_{k+1}\subset
\cdots \subset \bigcup_{k= 0}^{\infty}\mathcal{X}^{\Delta}_k=\mathcal{X}^{\Delta}.
\end{equation*}
}
\end{definition}
%%%%%
Since the map $h^{\Delta}$ is a proper map,
it extends the map
$h^{\Delta}_+:\mathcal{X}^{\Delta}_+\to Y_+$
between one-point compactifications, where
$X_+$ denotes the one-point compactification of a locally compact space
$X$.

%%%(Theorem 2.2)%%
\begin{theorem}[\cite{Va}, \cite{Va2} 
(cf. \cite{KY7}, \cite{Mo2})]\label{thm: simp}
%%%%%%%%
Let $h:X\to Y$ be a surjective map such that
$h^{-1}(y)$ is a finite set for any $y\in Y,$ 
$i:X\to \R^N$ an embedding, and let
$(\mathcal{X}^{\Delta},h^{\Delta})$ denote the simplicial resolution of $(h,i)$.
\par
\begin{enumerate}
%%(i)%%
\item[$\I$]
If $X$ and $Y$ are semi-algebraic spaces and the
two maps $h$, $i$ are semi-algebraic maps, then
$h^{\Delta}_+:\mathcal{X}^{\Delta}_+\stackrel{\simeq}{\rightarrow}Y_+$
is a homology equivalence.\footnote{%
%%%(FootNote 2)%%%%
It is known that the map $h^{\Delta}_+$ is a homotopy equivalence 
\cite[page 156]{Va2}.
(cf. \cite[Theorem in page 43]{GM}).
But in this paper we do not need such a stronger assertion.
}
%%(End of FootNote 2)%%%
Moreover,
there is an embedding $j:X\to \R^M$ such that
the associated simplicial resolution
$(\tilde{\mathcal{X}}^{\Delta},\tilde{h}^{\Delta})$ of $(h,j)$
is non-degenerate.
%%%
\par
%(ii)%%
\item[$\II$]
If there is an embedding $j:X\to \R^M$ such that its associated simplicial resolution
$(\tilde{\mathcal{X}}^{\Delta},\tilde{h}^{\Delta})$
is non-degenerate,
the space $\tilde{\mathcal{X}}^{\Delta}$
is uniquely determined up to homeomorphism and
there is a filtration preserving homotopy equivalence
$q^{\Delta}:\tilde{\mathcal{X}}^{\Delta}\stackrel{\simeq}{\rightarrow}{\mathcal{X}}^{\Delta}$ such that $q^{\Delta}\vert X=\mbox{id}_X$.
\qed
\end{enumerate}
%%%
\end{theorem}
\section{The Vassiliev spectral sequences}
\label{section: spectral sequence}
%%%%%%%%%%%%%%%%%

Our goal in this section is to construct, by means of the
{\it non-degenerate} simplicial resolutions  of the discriminants, two 
Vassiliev type spectral sequences converging to the homology of
$\Po^{d,m}_n$ and that of $\Hol_d^*(S^2,\CP^{mn-1})$,
respectively.

%%(Definition 3.1)%%%%%%%
\begin{definition}\label{Def: 3.1}
{\rm
%%%%%(i)%%%
\par
(i)
Let 
$\Sigma^{d,m}_n$ denote \emph{the discriminant} of
$\Po^{d,m}_n$ in $\P^d(\C)^m$ given by
the complement
%%%%%%%%%%%%%%%%%
\begin{eqnarray*}
%%%
\Sigma^{d,m}_n &=& 
\P^d(\C)^m\setminus \Po^{d,m}_n
\\
&=&
\{(f_1,\cdots ,f_{m})\in \P^d(\C)^m :
\textbf{\textit{f}}_1(x)=\cdots
=\textbf{\textit{f}}_m(x)=\mathbf{0}
\mbox{ for some }x\in \C\}.
%%%%%%%
\end{eqnarray*}
%%%%%%%%%%%%%%%%
Let us write $\P^d (m,n)=(\P^d(\C)^n)^m$, and
let $\tilde{\Sigma}^d$
be \emph{the discriminant} of $\Hol_d^*(S^2,\CP^{mn-1})$ in
$\P^d (m,n)$
given by
%%%%
\begin{eqnarray*}
\tilde{\Sigma}^d
&=&
\P^d (m,n)\setminus \Hol_d^*(S^2,\CP^{mn-1})
\\
&=&
\{(f_1,\cdots ,f_{mn})\in \P^d (m,n) :
f_1(x)=\cdots =f_{mn}(x)=0
\mbox{ for some }x\in \C\}.
\end{eqnarray*}
\par
%%(ii)%%
(ii)
Let
$Z^{d,m}_n\subset \Sigma^{d,m}_n\times\C$ denote 
{\it the tautological normalization} of $\Sigma^{d,m}_n$ 
given by
$$
Z^{d,m}_n=\{
((f_1(z),\cdots ,f_m(z),x)\in \Sigma^{d,m}_n\times\C:
\textit{\textbf{f}}_1(x)=\cdots =\textit{\textbf{f}}_m(x)=\mathbf{0}\}.
$$
Similarly,
let  $\tilde{Z}^d_N\subset \tilde{\Sigma}^d\times \C$
be 
{\it the tautological normalization} of 
 $\tilde{\Sigma}^d$ given by
$$
\tilde{Z}^d=\{((f_1,\ldots ,f_{mn}),x)\in \tilde{\Sigma}^d\times\C
:f_1(x)=\cdots =f_{mn}(x)=0\}.
$$
%%%
Projection on the first factor  gives the surjective maps
$\pi^{d,m}_n:Z^{d,m}_n\to \Sigma^{d,m}_n$
and
$\tilde{\pi}^{d}:\tilde{Z}^{d}\to \tilde{\Sigma}^{d}$,
respectively.
%%%
}
\end{definition}
%%%%%%

%%(Definition 3.2)%%%%
\begin{definition}\label{non-degenerate simp.}
%%%
(i)
Let $\varphi :\P^d(m,n)\stackrel{\cong}{\rightarrow}
\C^{dmn}$ be any fixed homeomorphism and define the embedding
$j_d:\tilde{Z}^d\to \C^{dmn+2d+1}$ by
%%(3.1)%%
\begin{equation}\label{3.1}
%%%%%%%%%%%%%%%
j_d((f_1,\cdots ,f_{mn}),x)=(\varphi (f_1,\cdots ,f_{mn}),
1,x,x^2,\cdots ,x^{2d})
%%%%%
\end{equation}
%%%%%
for $((f_1,\cdots ,f_{mn}),x)\in \tilde{Z}^d$.
%%%%
Similarly, define the embedding
$\tilde{i}:Z^{d,m}_n\to \C^{dmn+2d+1}$ by 
%%(3.2)%%
\begin{equation}\label{3.2}
%%%
\tilde{i}((f_1,\cdots ,f_m),x)=j((\textit{\textbf{f}}_1(z),\cdots ,\textit{\textbf{f}}_m(z),x)
%%%
\end{equation}
%%%%%
for $((f_1,\cdots ,f_m),x)\in Z^{d,m}_n$,
where
where $\textit{\textbf{f}}_k(z)$ denotes the $n$-tuple of  polynomials 
defined  in (\ref{equ: bff}).
Note that it is easy to see that the following holds:
%%%(3.3)%%%
\begin{equation}\label{equ: emb-discriminant}
%%%
j=\tilde{i}\circ (\hat{j}^{d,m}_n\times \mbox{id}_{\C})
\end{equation}
where
$\hat{j}^{d,m}_n:\Sigma^{d,m}_n\to \tilde{\Sigma}^d$ denote the embedding defined by
%%(3.4)%%
\begin{equation}\label{equ: 3.4}
%%%%
\tilde{j}^{d,m}_n(f_1(z),\cdots ,f_m(z))=
\big(\textit{\textbf{f}}_1(z),\cdots ,\textit{\textbf{f}}_m(z)\big)
\end{equation}
for $(f_1(z),\cdots ,f_m(z))\in\Sigma^{d,m}_n$.
%%%(ii)%%%
\par (ii)
Let 
$(\SZ,{\pi}^{\Delta}:\SZ\to\Sigma^{d,m}_n)$ 
and
$(\tilde{\mathcal{X}}^{d},\tilde{\pi}^{\Delta}:\tilde{\mathcal{X}}^{d}\to
\tilde{\Sigma}^d)$ 
be the  simplicial resolutions of $(\pi^{d,m}_n,i)$ and
 $(\tilde{\pi},j)$,
respectively.
%%%
Then it is easy to see that
$\mathcal{X}^d$ and $\tilde{\mathcal{X}}^d$ are no-degenerate simplicial
resolutions, and that
there are two
natural increasing filtrations
%%%%%
\begin{eqnarray*}
\emptyset 
&=&
\SZ_0
\subset \SZ_1\subset 
\SZ_2\subset \cdots \cdots\subset
\bigcup_{k= 0}^{\infty}\SZ_k=\SZ,
\\
\emptyset 
&=&
\tilde{\mathcal{X}}^{d}_0
\subset \tilde{\mathcal{X}}^{d}_1\subset 
\tilde{\mathcal{X}}^{d}_2\subset \cdots\cdots
\subset
\bigcup_{k= 0}^{\infty}\tilde{\mathcal{X}}^{d}_k
=\tilde{\mathcal{X}}^{d},
\end{eqnarray*}
%%%%%
such that
%%%%%(3.5)%%%
\begin{equation}\label{filt}
%%%%
\mathcal{X}^d_k=\mathcal{X}^d \quad\mbox{ if }k>
\Big\lfloor \frac{d}{n}\Big\rfloor
\quad \mbox{and}\quad
\tilde{\mathcal{X}}^d_k=\tilde{\mathcal{X}}^d\quad \mbox{ if }k>d.
\end{equation}
%%%%%
%%%%%%%%%%
\end{definition}
%%%%%%%%%

%\begin{rem}
%{\rm
%Note that our notation conflicts with the one used in the previous section for the simplicial resolution truncated after the $d$-th term.  Now the $(d)$ in $\SZ(d)$ refers not to truncation but to the homogeneous degree of the polynomials in $A_d^*$ etc.  We shall continue with this abuse of notation when we define truncated resolutions below.
%}
%\end{rem}

%%%%%%
%By Lemma \ref{lemma: simp} 
By Theorem \ref{thm: simp},
%\cite[Lemma 2.2]{KY4}, 
the map
%%%%%%%%%
$\pi_{+}^{\Delta}:\SZ_+\stackrel{\simeq}{\rightarrow}{\Sigma^{d,m}_{n+}}$
%%%%%
is a homology equivalence.
%%
%\par
%%%
Since
${\mathcal{X}_k^{d}}_+/{\SZ_{k-1}}_+
\cong (\SZ_k\setminus \SZ_{k-1})_+$,
we have a spectral sequence 
%%%%%%%%%%%
%$$
$$
\big\{E_{t;d}^{k,s},
d_t:E_{t;d}^{k,s}\to E_{t;d}^{k+t,s+1-t}
\big\}
\Rightarrow
H^{k+s}_c(\Sigma_n^{d,m},\Z),
$$
%$$
where
$E_{1;d}^{k,s}=\tilde{H}^{k+s}_c(\SZ_k\setminus\SZ_{k-1},\Z)$ and
$H_c^k(X,\Z)$ denotes the cohomology group with compact supports given by 
$
H_c^k(X,\Z)= H^k(X_+,\Z).
$
%%%%%%%%%%%%%
%\par\vspace{2mm}\par
%%%%%
\par
Since there is a homeomorphism
$\P^d(\C)^m\cong \C^{dm}$,
by Alexander duality  there is a natural
isomorphism
%%%(3.6)%%%
\begin{equation}\label{Al}
\tilde{H}_k(\Po^{d,m}_n,\Z)\cong
\tilde{H}_c^{2md-k-1}(\Sigma_n^{d,m},\Z)
\quad
\mbox{for any }k.
\end{equation}
%%%
By
reindexing we obtain a
spectral sequence
%%
%%%(3.7)%%%
\begin{eqnarray}\label{SS}
%%%%%%%%%%%%%%%%%%%
&&\big\{E^{t;d}_{k,s}, d^{t}:E^{t;d}_{k,s}\to E^{t;d}_{k+t,s+t-1}
\big\}
\Rightarrow H_{s-k}(\Po^{d,m}_n,\Z),
\end{eqnarray}
%%%%%%%
%if $s-k\leq 2nd-2$,
where
$E^{1;d}_{k,s}=
\tilde{H}^{2md+k-s-1}_c(\SZ_k\setminus\SZ_{k-1},\Z).$
%and $\tilde{E}^t_{r,s}(d)=E_t^{r,N_d^*-1-s}(d).$
%%%%%%
\par\vspace{3mm}\par
By a complete similar method we also have the spectral sequence
%%%(3.8)%%%
\begin{eqnarray}\label{SSSS}
%%%%%%%%%%%%%%%%%%%
&&\big\{\tilde{E}^{t;d}_{k,s}, \tilde{d}^{t}:
\tilde{E}^{t}_{k,s}\to \tilde{E}^{t}_{k+t,s+t-1}
\big\}
\Rightarrow H_{s-k}(\Hol_d^*(S^2,\CP^{mn-1}),\Z),
\end{eqnarray}
%%%%%%%
where
$\tilde{E}^{1}_{k,s}=
\tilde{H}^{2dmn+k-s-1}_c(\tilde{\mathcal{X}}^{d}_k\setminus
\tilde{\mathcal{X}}_{k-1}^{d},\Z).$
\par\vspace{2mm}\par
%%%

For a space $X$, let $F(X,k)\subset X^k$ denote the ordered
configuration space given by
$$
F(X,k)=\{(x_1,\cdots ,x_k)\in X^k:
x_i\not= x_j\mbox{ if }i\not= j\}.
$$
Let $S_k$ be the symmetric group of $k$ letters.
Then the group $S_k$ acts on $F(X,k)$ by permuting coordinates
and let $C_k(X)$ denote the orbit space
$C_k(X)=F(X,k)/S_k$.
%%%
Let $X^{\wedge k}$ denote the $k$-fold smash product of a space $X$,
i.e.
$X^{\wedge k}=X\wedge  \cdots \wedge X$
$(k$-times).
Then $S_k$ acts on $X^{\wedge k}$ by the coordinate permutation
and define  $D_k(X)$ by the equivariant half smash product
%%%%
%%%%(3.9)%%
\begin{equation}\label{equ: DkX}
D_k(X)=F(\C,k)_+\wedge_{S_k}X^{\wedge k}.
\end{equation}
%%%
In particular, for $X=S^1$ we set
%%(3.10)%%
\begin{equation}\label{equ: DkS}
D_k=D_k(S^1).
\end{equation}

%%%%(Lemma 3.3)%%%%
\begin{lemma}\label{lemma: vector bundle}
%%%%%%%%%%%%%%%%%%
%%%
If  
$1\leq k\leq \lfloor \frac{d}{n}\rfloor$,
$\SZ_k\setminus\SZ_{k-1}$
is homeomorphic to the total space of a real affine
bundle $\xi_{d,k}$ over $C_k(\C)$ with rank 
$l_{d,k}=2m(d-nk)+k-1$.
%%%%%%%%%%%%%%%%%%
\end{lemma}
%%%%%%%%(Proof of Lemma 3.3)%%%
\begin{proof}
%%%%%%%%%%%

The argument is exactly analogous to the one in the proof of  
\cite[Lemma 4.4]{AKY1}. 
%%%
%%%
Namely, an element of $\SZ_k\setminus\SZ_{k-1}$ is represented by
the $(m+1)$-tuple 
$(f_1(z),\cdots ,f_m(z),u)$, where 
$(f_1(z),\cdots ,f_m(z))$ is an $m$-tuple of monic polynomials of the same
degree $d$
in $\Sigma^{d,m}_n$ and $u$ is an element of the interior of
the span of the images of $k$ distinct points 
$\{x_1,\cdots, x_k\}\in C_k(\C)$ 
such that
%%%%%
$\{x_j\}_{j=1}^k$ are common roots of 
$\{f_i(z)\}_{k=1}^m$ of multiplicity $n$
%%%%
under a suitable embedding.
%of $\C$ into Euclidean space.% satisfying the condition ({\ref{equ: filtration}}$)_k$.
%%%%
\ 
Since the $k$ distinct points $\{x_j\}_{j=1}^k$ 
are uniquely determined by $u$, by the definition of the non-degenerate simplicial resolution,
there are projection maps
%%%%%%%%%%%%%
%%%()%%%%%
%%%%%%%%%%%%%
%\begin{equation}\label{pik}
%%%%%%%%%%%% 
$\pi_{k,d} :{\cal X}^{d}_k\setminus
{\cal X}^{d}_{k-1}\to C_{k}(\C)$
%\end{equation}
%%%%%%%%%%%%
defined by
$((f_1,\cdots ,f_m),u) \mapsto 
\{x_1,\cdots ,x_k\}$. 

\par
%%%
%%%%%(Fiber of pi_k)%%%%%%
Now suppose that $1\leq k\leq \lfloor \frac{d}{n}\rfloor$
and $1\leq i\leq m$.
Let $c=\{x_j\}_{j=1}^k\in C_{k}(\C)$
 be any fixed element and consider the fibre  $\pi_{k,d}^{-1}(c)$.
%%%
If we
consider the condition
that a polynomial $f_i(z)\in\P^d(\C)$
is divided by
the polynomial
$\prod_{j=1}^k(z-x_j)^n$,
then it is easy to see that this condition is equivalent
to the following condition:
%%%  
%%%(3.11)%%%
\begin{equation}\label{equ: equation}
%%%%%%%%%%%
f^{(t)}_i(x_j)=0
\quad
\mbox{for }0\leq t<n,\ 1\leq j\leq k.
\end{equation}
%%%%%%%%%%%
%%
In general, for each $0\leq t< n$ and $1\leq j\leq k$,
the condition $f^{(t)}_i(x_j)=0$ 
gives
one  linear condition on the coefficients of $f_i(z)$,
and it determines an affine hyperplane in $\P^d(\C)$. 
%%%
For example, if we set $f_i(z)=z^d+\sum_{s=1}^da_{s}z^{d-s}$,
then
$f_i(x_j)=0$ for all $1\leq j\leq k$
if and only if
%%%%%%
%%()(matrix equation)%%
\begin{equation*}\label{equ: matrix equation}
%%%%%%%%%
\begin{bmatrix}
1 & x_1 & x_1^2 & x_1^3 & \cdots & x_1^{d-1}
\\
1 & x_2 & x_2^2 & x_2^3 & \cdots & x_2^{d-1}
\\
\vdots & \ddots & \ddots & \ddots & \ddots & \vdots
%\\
%1 & x_{k-1} & x_{k-1}^2 & \cdots & x_{k-1}^{d-1}
\\
1 & x_k & x_k^2 & x_k^3 & \cdots & x_k^{d-1}
\end{bmatrix}
%%%%
\cdot
\begin{bmatrix}
a_{d}\\ a_{d-1} \\ \vdots %\\ a_{2,t} 
\\ a_{1}
\end{bmatrix}
=
-
\begin{bmatrix}
x_1^d\\ x_2^d \\ \vdots %\\ s_{t,k-1}-x_{k-1}^d 
\\ x_k^d
\end{bmatrix}
\end{equation*}
%%%%
%%%%
%%%%
Similarly, $f^{\p}_i(x_j)=0$ for all $1\leq j\leq k$
if and only if
%%()(matrix equation)%%
\begin{equation*}\label{equ: matrix equation2}
%%%%%%%%%
\begin{bmatrix}
0 &1 & 2x_1 & 3x_1^2 & \cdots & (d-1)x_1^{d-2}
\\
0 & 1 & 2x_2 & 3x_2^2 & \cdots & (d-1)x_2^{d-2}
\\
\vdots & \vdots & \ddots & \ddots & \ddots & \vdots
%\\
%1 & x_{k-1} & x_{k-1}^2 & \cdots & x_{k-1}^{d-1}
\\
0 &1 & 2x_k & 3x_k^2 & \cdots & (d-1)x_k^{d-2}
\end{bmatrix}
%%%%
\cdot
\begin{bmatrix}
a_d \\
 a_{d-1} \\ \vdots %\\ a_{2,t} 
\\ a_{1}
\end{bmatrix}
=
-
\begin{bmatrix}
dx_1^{d-1}\\ dx_2^{d-1} \\ \vdots %\\ s_{t,k-1}-x_{k-1}^d 
\\ dx_k^{d-1}
\end{bmatrix}
\end{equation*}
%%%%
and
$f^{\p\p}_i(x_j)=0$ for all $1\leq j\leq k$
if and only if
%%()(matrix equation)%%
\begin{equation*}\label{equ: matrix equation2}
%%%%%%%%%
\begin{bmatrix}
0 & 0 & 2 & 6x_1 & \cdots & (d-1)(d-2)x_1^{d-3}
\\
0 & 0 & 2 & 6x_2 & \cdots & (d-1)(d-2)x_2^{d-3}
\\
\vdots & \vdots & \ddots & \ddots & \ddots & \vdots
%\\
%1 & x_{k-1} & x_{k-1}^2 & \cdots & x_{k-1}^{d-1}
\\
0 & 0 & 2 & 6x_k & \cdots & (d-1)(d-2)x_k^{d-3}
\end{bmatrix}
%%%%
\cdot
\begin{bmatrix}
a_d \\
 a_{d-1} \\ \vdots %\\ a_{2,t} 
\\ a_{1}
\end{bmatrix}
=
-
\begin{bmatrix}
d(d-1)x_1^{d-2}\\ 
d(d-1)x_2^{d-1} 
\\ \vdots %\\ s_{t,k-1}-x_{k-1}^d 
\\ d(d-1)x_k^{d-2}
\end{bmatrix}
\end{equation*}
%%%%
%%%%
%%%%
and so on.
Since $1\leq k\leq \lfloor \frac{d}{n}\rfloor$ and
 $\{x_j\}_{j=1}^k\in C_k(\C)$, it follows from
the properties of Vandermonde matrices 
that the the condition  
(\ref{equ: equation}) 
gives exactly $nk$ affinely independent conditions on the coefficients of $f_i(z)$.
Hence,
we see that 
the space of $m$-tuples $(f_1(z),\cdots ,f_m(z))\in\P^d(\C)^m$ 
of monic polynomials which satisfies
the condition (\ref{equ: equation}) for each $1\leq i\leq m$
is the intersection of $mnk$ affine hyperplanes in general position, and
it has codimension $mnk$ in $\P^d(\C)^m$.
%%%%%
Therefore,
the fibre $\pi_{k,d}^{-1}(c)$ is homeomorphic  to the product of an open $(k-1)$-simplex
 with the real affine space of dimension
 $2m(d-nk)$. 
 Because we can also check that the local triviality holds,
 we see that
$\SZ_k\setminus\SZ_{k-1}$ is a real affine bundle over $C_{k}(\C)$ of rank $l_{d,k}
=2m(d-nk)+k-1$.
%%%%%%%%%
\end{proof}
%%(End of proof of Lemma 3.3)%%%

\par\vspace{2mm}\par

By using a complete similar method of Lemma
\ref{lemma: vector bundle} we can also prove 
the following  result.
%%%%(Lemma 3.4)%%%%
\begin{lemma}\label{lemma: vector bundle*}
%%%%%%%%%%%%%%%%%%
%%%
If  
$1\leq k\leq  d$,
$\tilde{\mathcal{X}}^{d}_k\setminus
\tilde{\mathcal{X}}^{d}_{k-1}$
is homeomorphic to the total space of a real affine
bundle $\tilde{\xi}_{d,k}$ over $C_k(\C)$ with rank 
$\tilde{l}_{d,k}=2mn(d-k)+k-1$.
\qed
%%%%%%%%%%%%%%%%%%
\end{lemma}
%%%%%%%%(End of Lemma 3.4)%%%

%%%%%%(Lemma 3.5)%%
\begin{lemma}\label{lemma: E1}
%%%%%%
If $1\leq k\leq  \lfloor \frac{d}{n}\rfloor$, there is a natural isomorphism
$$
E^{1;d}_{k,s}\cong
H_{s-2(mn-1)k}(C_{k}(\C),\pm \Z),
$$
where %we set $r_{\rm min}=r_{\rm min}(I)$ and
the twisted coefficients system $\pm \Z$  comes from
the Thom isomorphism.\footnote{%
%%%%(FootNote 3)%%%%%%%%
The twisted coefficients system $\pm \Z$ on
$C_k(\C)$ is induced by the sign representation of the symmetric group. 
(cf. \cite[page 114 and 254]{Va}).}
%%%(End of FootNote 3)%%%%
%%%
%%%%%(End of Lemma 3.5)%%%
\end{lemma}
%%%%
\begin{proof}
%%%(Proof of Lemma 3.5)%%
Suppose that $1\leq k\leq \lfloor \frac{d}{n}\rfloor$.
%%%
By Lemma \ref{lemma: vector bundle}, there is a
homeomorphism
$
(\SZ_k\setminus\SZ_{k-1})_+\cong T(\xi_{d,k}),
$
where $T(\xi_{d,k})$ denotes the Thom space of
%one-point compactification of
$\xi_{d,k}$.
%%%%%%%
Since $(2md+k-s-1)-l_{d,k}
=
2mnk-s,$
by using the Thom isomorphism and the Poincar\'e duality,
there is a natural isomorphism 
%%%
$
E^{1;d}_{k,s}
\cong 
\tilde{H}^{2md+k-s-1}(T(\xi_{d,k}),\Z)
\cong
H^{2mnk-s}_c(C_{k}(\C),\pm \Z)
\cong H_{s-2(mn-1)k}(C_k(\C),\pm \Z)
$
and this completes the proof.
\end{proof}
%%%(End of proof of Lemma 3.5)%%%%%%

A similar method also proves the following:
%%%%%%(Lemma 3.6)%%
\begin{lemma}\label{lemma: E1*}
%%%%%%
If $1\leq k\leq  d$, there is a natural isomorphism
$$
\tilde{E}^{1}_{k,s}\cong
H_{s-2(N-1)k}(C_{k}(\C),\pm \Z).
\qed
$$
\end{lemma}
%%%%(End of Lemma 3.6)%%%%%%%

%%%(Corollary 3.7)%%%
\begin{corollary}\label{crl: Er}
\begin{enumerate}
\item[$\I$]
There is a natural isomorphism
$$
E^{1;d}_{k,s}\cong
\begin{cases}
\Z & \mbox{if }(k,s)=(0,0)
\\
H_{s-2(mn-1)k}(C_k(\C),\pm\Z) &
\mbox{if }1\leq k\leq \lfloor \frac{d}{n}\rfloor ,
\ s\geq 2(mn-1)k
\\
0 & \mbox{otherwise}
\end{cases}
$$
%%(ii)%%%
\item[$\II$]
There is a natural isomorphism
$$
\tilde{E}^{1}_{k,s}\cong
\begin{cases}
\Z & \mbox{if }(k,s)=(0,0)
\\
H_{s-2(mn-1)k}(C_k(\C),\pm\Z) &
\mbox{if }1\leq k\leq d ,
\ s\geq 2(mn-1)k
\\
0 & \mbox{otherwise}
\end{cases}
$$
\end{enumerate}
\end{corollary}
%%%%(End of Corollary 3.7)
\begin{proof}
%%%(Proof of Corollary 3.7)%%
It is easy to see that
the assertion (i)  follows from Lemma \ref{lemma: E1} and (\ref{filt}).
The assertion (ii) follows similarly.
%%%
\end{proof}
%%%(End of Proof of Corollary 3.7)

%%(Remark 3.8)%%%
\begin{remark}\label{rmk: spectral sequence for general N}
%%%%%%%%%
By using the complete similar way, for an integer $N\geq 2$ one can obtain the spectral sequence
$$
\{\tilde{E}^{t;N}_{k,s},d^t:\tilde{E}^{t;N}_{k,s}
\to 
\tilde{E}^{t;N}_{k+t,s+t-1}\}
\Rightarrow H_{s-k}(\Hol_d^*(S^2,\CP^{N-1}),\Z)
$$
such that
%%(3.12)%%
\begin{equation}\label{equ: Er in general}
%%%%%%%%
\tilde{E}^{1;N}_{k,s}\cong
\begin{cases}
\Z & \mbox{if }(k,s)=(0,0)
\\
H_{s-2(N-1)k}(C_k(\C),\pm\Z) &
\mbox{if }1\leq k\leq d ,
\ s\geq 2(N-1)k
\\
0 & \mbox{otherwise}
\quad
\qed
\end{cases}
\end{equation}
%%%%
\end{remark}
%%%(End of Remark 3.8)%%%%
%%
%%%(Remark 3.9)%%%
\begin{remark}
%%%%%%%%%%%%%%%%%%
One can show that for $1\leq k\leq \lfloor \frac{d}{n}\rfloor$ there is a homeomorphism
%%%
\begin{equation}
%%%
(\mathcal{X}^d_k\setminus \mathcal{X}^d_{k-1})\times \C^{dm(n-1)}
\cong
\tilde{\mathcal{X}}^d_k\setminus \tilde{\mathcal{X}}^d_{k-1}.
\end{equation}
%%%%
Hence, it is easy to see that there is an isomorphism
$E^{1;d}_{k,s}\cong \tilde{E}^1_{k,s}$ for any $s$ if
$1\leq k\leq \lfloor \frac{d}{n}\rfloor$
(but this is easily seen by Lemma \ref{lemma: E1} and Lemma \ref{lemma: E1*}).
\qed
%%%%%%%%%%%%
\end{remark}

%%(SECTION 4)%%%
\section{Stabilization maps}\label{section: sd}
%%%%%%%%%%%%%%%

%%%(Definition 4.1)%%%%%%%
\begin{definition}\label{def: stabilization}
%%%%%%%%%%%%%%%%%%%%%%%%%%
Let 
%%%%(4.1)%%%
\begin{equation}\label{equ: stabilization map}
%%%%%%%%%%%
s^{d,m}_n:\Po^{d,m}_n\to \Po^{d+1,m}_n
\end{equation}
%%%%
 denote
the stabilization map given by adding the points from the infinity
as in \cite[\S 5 page 57]{Se}.
Similarly, let
%%(4.2)%%
\begin{equation}
s_d:\Hol_d^*(S^2,\CP^{mn-1})\to \Hol_{d+1}^*(S^2,\CP^{mn-1})
\end{equation}
%%%%
be the stabilization map given in \cite{Se}.
\qed
%%%%%%
\end{definition}
%%%%%%%%(End of definition 4.1)%%%%
%\par\vspace{2mm}\par
%%%%%%%%%
It is easy to see that the following diagram is commutative
(up to homotopy)
%%(4.3)%%%
\begin{equation}\label{CD: stab0}
\begin{CD}
%%%%%%
\Po^{d,m}_n @>s^{d,m}_n>> \Po^{d+1,m}_n
\\
@V{j^{d,m}_n}VV @V{j^{d+1,m}_n}VV
\\
\Hol_d^*(S^2,\CP^{mn-1}) @>s_d>>\Hol_{d+1}^*(S^2,\CP^{mn-1})
\end{CD}
\end{equation}
%%%%%%%%%%%%
Note that
the map $s^{d,m}_n$ clearly extends to the map
$\P^d(\C)^m\to\P^{d+1}(\C)^m$ and
its restriction gives the stabilization map
%%%%%
$\tilde{s}^{d,m}_n:\Sigma^{d,m}_n\to \Sigma^{d+1,m}_n$
between discriminants.
It is easy to see that it  also extends to the open embedding
%%%(4.4)%%
\begin{equation}\label{equ: open-stab}
%%%%%%%
\tilde{s}^{d,m}_n:\Sigma^{d,m}_n\times\C^m\to \Sigma^{d+1,m}_n.
\end{equation}
%%%%%%%%%%%%%
%%%%
Since the one-point compactification is contravariant for open embeddings,
it
induces the map
%%(4.5)%%
\begin{equation}\label{equ: embedding3}
%%%%%%%%%%%%%%%
\tilde{s}^{d,m}_{n+}: 
(\Sigma^{d+1,m}_{n})_+
\to
(\Sigma^{d,m}_n\times \C^{m})_+=\Sigma^{d,m}_{n+}\wedge S^{2m}
\end{equation}
%%%%%%
between one-point compactifications.
Then we obtain  the following  diagram is commutative
%%%(4.6)%%
\begin{equation}\label{diagram: discriminant}
%%%%%%%%%%%%
\begin{CD}
\tilde{H}_k(\Po^{d,m}_n,\Z) 
@>{s^{d,m}_{n*}}>>\tilde{H}_k(\Po^{d+1,m}_n,\Z)
\\
@V{Al}V{\cong}V @V{Al}V{\cong}V
\\
\tilde{H}^{2dm-k-1}_c(\Sigma^{d,m}_n,\Z) 
@>\tilde{s}^{d,m*}_{n+}>>
\tilde{H}^{2(d+1)m-k-1}_c(\Sigma^{d+1,m}_n,\Z)
\end{CD}
%%%%%%%%%%
\end{equation}
%%%% 
where $Al$ denotes the Alexander duality isomorphism and
 $\tilde{s}^{d,m*}_{n+}$ denotes the composite of the
%homomorphisms 
the suspension isomorphism with the homomorphism
${(\tilde{s}^{d,m}_n)^*}$,
%{\small
$$
%\begin{CD}
\tilde{H}^{2dm-k-1}_c(\Sigma^{d,m}_n)
\stackrel{\cong}{\rightarrow}
\tilde{H}^{2(d+1)m-k-1}_c(\Sigma^{d,m}_n\times \C^{m})
\stackrel{(\tilde{s}^{d,m}_{n+})^*}{\longrightarrow}
\tilde{H}^{2(d+1)m-k-1}_c(\Sigma^{d+1,m}_{n}).
%\end{CD}
$$
%}
%%%%%
%\newline
Note that the map $\tilde{s}^{d,m}_n$  induces the filtration preserving
map 
%%%(4.7)%%
\begin{equation}
%%%%%%%
\hat{s}^{d,m}_n:\SZ \times \C^m\to \mathcal{X}^{d+1}
\end{equation}
%%%%
and
it defines the homomorphism of spectral sequences
%%%%%%%%%%%
%%%(4.8)%%%%
\begin{equation}\label{equ: theta1}
%%%%%%%%%%%%
\{ \theta_{k,s}^t:E^{t;d}_{k,s}\to E^{t;d+1}_{k,s}\}.
\end{equation}
%%%
%%%
%%%
%%%(Lemma 4.2)%%%
\begin{lemma}\label{lmm: E1}
%%%%%%%%%%%%%%%%%
If $1\leq k\leq \lfloor \frac{d}{n}\rfloor$, 
$\theta^1_{k,s}:
E^{1;d}_{k,s}\stackrel{\cong}{\rightarrow} 
{E}^{1;d+1}_{k,s}$ is
an isomorphism for any $s$.
\end{lemma}
%%%%%%%%%%%%%%%%
\begin{proof}
%%%%%%%%%%%%%%%%
Suppose that $1\leq k\leq \lfloor \frac{d}{n}\rfloor$.
If we set 
$\hat{s}^{d,m}_{n;k}=\hat{s}^{d,m}_n\vert
\SZ_k\setminus\SZ_{k-1},$
the diagram
$$
\begin{CD}
%\C^n\times 
(\SZ_k\setminus\SZ_{k-1})\times \C^{m} @>\pi_{k,d}>> C_{k}(\C)
\\
@V{\hat{s}^{d,m}_{n;k}}VV \Vert @.
\\
\mathcal{X}^{d+1}_k\setminus 
\mathcal{X}^{d+1}_{k-1} 
@>\pi_{k,d+1}>> C_{k}(\C)
\end{CD}
$$
is commutative.
Hence, $\theta^1_{k,s}$ is an isomorphism.
%%%%%%%%
\end{proof}
%%(End of proof of Lemma 4.2)%%%%

Now we prove the main key result.
%%%(Theorem 4.3)%%
\begin{theorem}\label{thm: stab1}
%%%%%%%%%%%%%%%%
If $n\geq 2$,
the stabilization map 
$$
s^{d,m}_n:\Po^{d,m}_n\to
\Po^{d+1,m}_n
$$ 
is a homology equivalence for 
$\lfloor \frac{d}{n}\rfloor =\lfloor\frac{d+1}{n}\rfloor$, and it is a
homology equivalence through dimension
$D(d;m,n)$ for
$\lfloor \frac{d}{n}\rfloor <\lfloor\frac{d+1}{n}\rfloor$.
%%%%%%%%%%%%%%
\end{theorem}
%%%%%%%%%%%%%%
%%%(Proof of Theorem 4.3)%%%%
\begin{proof}
%%%%%%%%%%%%%%%%%%%%%%%%%%%%%
First, consider the case 
$\lfloor \frac{d}{n}\rfloor =\lfloor\frac{d+1}{n}\rfloor$.
In this case, by using Corollary \ref{crl: Er} and 
Lemma \ref{lmm: E1} it is easy to see that
$\theta^{1}_{k,s}:E^{1;d}_{k,s}\stackrel{\cong}{\longrightarrow}
E^{1;d+1}_{k,s}$ 
is an isomorphism for any $(k,s)$.
Hence, $\theta^{\infty}_{k,s}$ 
is an isomorphism for any $(k,s).$
Since $\theta^t_{k,s}$ is induced from $\hat{s}^{d,m}_n$,
it follows from (\ref{diagram: discriminant}) that
the map $s^{d,m}_n$ is a homology equivalence.
%%%
\par
%%%
Next assume that
$\lfloor \frac{d}{n}\rfloor <\lfloor\frac{d+1}{n}\rfloor$,
i.e. $\lfloor \frac{d+1}{n}\rfloor =\lfloor\frac{d}{n}\rfloor +1.$
In this case,
by considering the differential
$d^t:E^{t;d+\epsilon}_{k,s}\to E^{t;d+\epsilon}_{k+t,s+t-1}$
($\epsilon \in \{0,1\})$, Lemma \ref{lmm: E1}
and Corollary \ref{crl: Er}, we easily see that
$\theta^t_{k,s}:E^{t;d}_{k,s}\to E^{t;d+1}_{k,s}$ 
is an isomorphism for any $(k,s)$ and $t$ as long as
the condition $s-t\leq D(d;m,n)$ is satisfied.
Hence, if $s-t\leq D(d;m,n)$,
$\theta^{\infty}_{k,s}$ is always an isomorphism and so that
the map $s^{d,m}_n$ is a homology equivalence
through dimension $D(d;m,n)$.
%%%%%%%%%%%%%%%%%%%%%%%%%%%%%%%%
%%%%
\end{proof}
%%%(End of proof of Theorem 4.3)%%%%%%
%%%%%%%%%%%%%%%%%%%%%%%%%%%%%%%%%%%%%%%

%%(Theorem 4.4)%%%
\begin{theorem}[\cite{KY6}, Theorem 2.8]\label{thm: stab2}
%%%%%%%%%%%%%%%%%%%
If $n\geq 2$, the stabilization map
$$
s_d:\Hol_d^*(S^2,\CP^{mn-1})\to
\Hol_{d+1}^*(S^2,\CP^{mn-1})
$$ is a homology equivalence through dimension
$(2mn-3)(d+1)-1$.
\qed
\end{theorem}
%%%%%%%%%%%%%%%%%

%%(Definition 4.5)%%%
\begin{definition}\label{def: stability}
%%%%%%%%%%%%%%%%%%%%%
Let $\Po^{\infty, m}_n$ denote the colimit
$\dis \lim_{d\to\infty}\Po^{d,m}_n$
taken from the stabilization maps $s^{d,m}_n$'s.
Then the natural map
$i^{d,m}_n:\Po^{d,m}_n\to \Omega^2_d\CP^{mn-1}$
(given in (\ref{equ: jet inclusion})) induces the map
%%
%%%(4.9)%%
\begin{equation}\label{equ: natural map}
%%%%%%%%
i^{\infty,m}_n=
\varinjlim
 i^{d,m}_n:\Po^{\infty, m}_n\to
\lim_{d\to\infty}\Omega^2_d\CP^{mn-1}\simeq
\Omega^2S^{2mn-1}.
\end{equation}
%%%%%%%%%
%%%%
\end{definition}
%%(End of Definition 4.5)%%%%%

%%%%
Then we have the following result whose proof is given in the next section.

%%(Theorem 4.6)%%
\begin{theorem}\label{thm: natural map}
%%%%%%%%%%%%%%%%
If $n\geq 2$, the map
$
i^{\infty,m}_n:\Po^{\infty, m}_n
\stackrel{\simeq}{\longrightarrow}
\Omega^2S^{2mn-1}
$
is a homology equivalence.
%%%%%%%%%%%%%%%
\end{theorem}
%%%%%%%%%%%%%
%The proof of Theorem \ref{thm: natural map} is
%given in the next section. 
%postponed 
%until \S \ref{section: scanning maps}
%and give the proofs of the main results in the next section.

%%%(SECTION 5)%%%
\section{Scanning maps  and the unstable results}
\label{section: scanning maps}
%%%%%%%%%%%%%%%%%

In this section, we prove Theorem \ref{thm: natural map}
by using the scanning maps.
Next we give the proofs of the stability results
(Theorem \ref{thm: I} and Corollary \ref{thm: II}).

%%(Definition 5.1)%%%
\begin{definition}
%%%%%%
For a space $X$ let $\SP^d(X)$ denote the $d$-th symmetric product
defined by the quotient space $\SP^d(X)=X^d/S_d$, where
the symmetric group $S_d$ of $d$-letters acts on $X^d$ by the
permutation of coordinates.
Note that there is a natural inclusion
$C_d(X)\subset \SP^d(X)$.
%%%%%%%%%%%
\end{definition}
%%(End of definition 5.1)%%

%%(Remark 5.2)%%
\begin{remark}\label{rmk: notation}
%%%
(i)
An element of $\SP^d(X)$ may be identified with
the formal linear combination
$\alpha =\sum_{i=1}^kd_kx_i$
($\{x_i\}\in C_k(X), \ \sum_{k=1}^ld_k=d$).
We shall refer to $\alpha$
as configuration of points, the point $x_i$ having a multiplicity
$d_i$.
\par
(ii)
If $X =\C$, then  $\P^d(\C)$ can be easily identified with the space $\SP^d(\C)$ by the correspondence
$\prod_{i=1}^k(z-\alpha_i)^{d_i}\mapsto \sum_{i=1}^{k}d_i\alpha_i$.
\qed
%%%%%%%%%%%%%%%%%%%%%%%%%
\end{remark}
%%%(End of Remark 5.2)%%%

%%%(Definition 5.3)%%%
\begin{definition}
%%%
For a space $X$, define the space $\po^{d,m}_n(X)$ by
%%(5.1)%%
\begin{equation}\label{equ: Poly(X)}
\po^{d,m}_n(X)=
\{(\xi_1,\cdots ,\xi_m)\in\SP^d(X)^m:
(*) \}, 
\end{equation}
%%%%%%%
where the condition $(*)$ is given by
\begin{enumerate}
\item[$(*)$]
 $\cap_{i=1}^m\xi_i$ does not contain any element of multiplicity $\geq n$.
\end{enumerate}
%%%%
\end{definition}
%%%(End of Definition 5.3)%%%

%%%(Remark 5.4)%%%
\begin{remark}\label{rmk: Poly identification}
%%%%%%%%%%
By identifying $\P^d(\C)= \SP^d(\C)$ as in Remark \ref{rmk: notation}, 
we easily see that there is a homeomorphism
%%(5.2)%%
\begin{equation} \label{eq: Po}
\Po^{d,m}_n\cong\po^{d,m}_n(\C).
%%%%%%%%
\end{equation}
\end{remark}
%%%%%%%%%%%%%%%%

%%(Definition 5.5)%%%
\begin{definition}
%%%%%%%%%%%%%%%%
If $A\subset X$ is a closed subspace, we define
%%(5.3)%%
\begin{equation}
%%%%%%%%
\po^{d,m}_n(X,A)=\po^{d,m}_n(X)/\sim ,
\end{equation}
%%%%%
where the equivalence relation \lq\lq$\sim$\rq\rq
is defined by
$$
(\xi_1,\cdots ,\xi_m)\sim
(\eta_1,\cdots ,\eta_m)
\quad \mbox{ if }\quad \xi_i\cap (X\setminus A)=\eta_i \cap (X\setminus A)
$$
for each $1\leq i\leq m$.
Therefore, points in $A$ are \lq\lq ignored\rq\rq .
When $A\not=\emptyset$,
there is a natural inclusion
$$
\po^{d,m}_n(X,A)\subset \po^{d+1,m}_n(X,A)
$$
by adding points in $A$.
Define the space $\po^{m}_n(X,A)$ by the union
%%%(5.4)%%
\begin{equation}
%%%%%%%%%%
\po^{m}_n(X,A)=\bigcup_{d\geq 1}\po^{d,m}_n(X,A).
\end{equation}
%%%%%%%
\end{definition}
%%(End of definition 5.5)%%%%

%%(Remark 5.6)%%%%
\begin{remark}
%%%%
As a set, $\po^{m}_n(X,A)$ is bijectively equivalent to
the disjoint union
$\dis \bigcup_{d\geq 1}\po^{d,m}_n(X\setminus A)$.
But these two spaces are not homeomorphic.
For example, if $X$ is connected, then $\po^{m}_n(X,A)$ is connected.
\qed
%%%%%%%%%%%%%%%%
\end{remark}
%%%%%%%%(End of Remark 5.6)%%%

We need two kinds of scanning maps.
First, we define the scanning map for configuration space of
particles.

%%(Definition 5.7)%%
\begin{definition}
%%%%%%%%%%%%%%%
Let
us identify
$D^2=\{x\in\C :\vert x\vert \leq 1\}$, and
let $\epsilon >0$ be a fixed sufficiently small positive
number
Then
for each $w\in\C$, let $U_{w}$ be the open set
%%()%%
%\begin{equation}
%%%%
$U_{w}=\{x\in\C: \vert x-w\vert <\epsilon\}.$
%\end{equation}
%%%%%%%%
%%
Now define the scanning map
%%%(5.5)%%
\begin{equation}
sc^{d,m}_n:\po^{d,m}_n(\C)\to \Omega^2\po^{m}_n(D^2,S^1)
\end{equation}
%%%
as follows.
Let $\alpha =(\xi_1,\cdots ,\xi_m)\in \po^{d,m}_n(\C).$
Then let
$$
sc^{d,m}_n(\alpha):S^2=\C\cup\infty \to \po^{m}_n(D^2,S^1)
$$
denote the map given by
$$
w\mapsto
(\xi_1\cap\overline{U}_{w},\cdots ,\xi_m\cap\overline{U}_{w})
\in\po^{m}_n(\overline{U}_{w},\partial \overline{U}_{w})
\cong \po^{m}_n(D^2,S^1)
$$
for $w\in \C$,
where we use the canonical identification
$(\overline{U}_w,\partial \overline{U}_w)\cong (D^2,S^1)$.
%%%%%
Since $\dis\lim_{w\to\infty}sc^{d,m}_n(\alpha)(w)=(\emptyset ,\cdots ,\emptyset)$,
we set $sc^{d,m}_n(\alpha)(\infty)=(\emptyset ,\cdots ,\emptyset)$
and
we obtain the based map
$sc^{d,m}_n(\alpha)\in \Omega^2\po^{m}_n(D^2,S^1).$
%%%
\par
Since the space $\po^{d,m}_n(\C)$ is connected, the image of $sc^{d,m}_n$ is contained
in some path-component of
$\Omega^2\po^{m}_n(D^2,S^1),$
which is denoted by
%%()%%
%\begin{equation}
$\Omega^2_d\po^{m}_n(D^2,S^1).$
%\end{equation}
%%%%%%%%
Hence,  finally we obtain the map
%%%%%%%%%%
%%%(5.6)%%
\begin{equation}
%%%%%%%%%%
sc^{d,m}_n:
\po^{d,m}_n(\C)\to 
\Omega^2_d\po^{m}_n(D^2,S^1).
%%%%%%%%%
\end{equation}
%%%%%%%%%
Now we identify
$\Po^{d,m}_n\cong \po^{d,m}_n(\C)$ as in (\ref{eq: Po})
and
by setting $\dis S=\lim_{d\to\infty}sc^{d,m}_n$,
we obtain {\it the scanning map}
%%(5.7)%%%%%
\begin{equation}
%%%%%%%%%%%%
S:
\Po^{\infty,m}_n
\to \lim_{d\to\infty}\Omega^2_d
\po^{m}_n(D^2,S^1)\simeq \Omega^2_0
\po^{m}_n(D^2,S^1).
\end{equation}
%%%%
\end{definition}
%%(End of Definition 5.7)%%%

%%(Theorem 5.8)%%%
\begin{theorem}[\cite{GKY2}, \cite{Se}]
\label{thm: scanning map}
If $n\geq 2$, the scanning map
$$
S=\lim_{d\to\infty}sc^{d,m}_n:\Po^{\infty,m}_n
\stackrel{\simeq}{\longrightarrow}
\Omega^2_0\po^{m}_n(D^2,S^1)
$$
is a homology equivalence.
\end{theorem}
%%%%%%%%%%%%
\begin{proof}
%%%
The proof is similar to the argument of \cite[\S 3]{Se}.
Alternatively, we can prove this by using the complete similar way as in
\cite[page 99-100]{GKY2}
%%%
\end{proof}
%%(End of proof of Theorem 5.8)%%%%%

%%(Definition 5.9)%%%
\begin{definition}\label{def: 5.9}
%%%%%%%%%%%%%%%%%%%%
(i)
Let $\mathcal{P}^d(\C)$ denote the space of 
(not necessarily monic) all
polynomials $f(z)=\sum_{i=0}^da_iz^i\in\C[z]$ of degree exactly $d$
and let $\mathcal{P}oly^{d,m}_n$ denote the space of all
$m$-tuples $(f_1(z),\cdots ,f_m(z))\in \mathcal{P}^d(\C)^m$
such that polynomials $\{f_1(z),\cdots ,f_m(z)\}$ have no common root of
multiplicity $\geq n$.
\par
(ii)
For each nonempty open set $X\subset \C$, let
$\mathcal{P}oly^{m}_n(X)$ denote the space of
all $m$-tuples $(f_1(z),\cdots ,f_m(z))$ satisfying the following
two conditions:
%%(5.9.1), (5.9.2)%%%%%%%
\begin{enumerate}
%%%%%%%%%
\item[(\ref{def: 5.9}.1)]
$f_i(z)\in\C[z]$ is a complex polynomial of the same degree 
and it is not identically zero
for each $1\leq i\leq m$.
%%%(5.9.2)%%%%%%
\item[(\ref{def: 5.9}.2)]
%%%%%%%%%%
Polynomials $\{f_1(z),\cdots ,f_m(z)\}$ have no common root in $X$ of multiplicity
$\geq n$.
\end{enumerate}
%%%%%%
%%%%%%
When $X=\C$, we write $\mathcal{P}oly^{d,m}_n=\mathcal{P}oly^{d,m}_n(\C)$.
\qed
%%%%
\end{definition}
%%%(End of Definition 5.9)%%%%

%%(Remark 5.10)%%
\begin{remark}
%%%%%%%%%%%%%%%%
(i)
Note that $\mathcal{P}oly^{m}_n(\C)$ is bijectively equivalent to
the union $\bigcup_{d\geq 0}\mathcal{P}oly^{d,m}_n(\C)$, but these spaces are not homeomorphic because $\mathcal{P}oly^{m}_n(\C)$ is connected.
\par
(ii) It is easy to see that there are homeomorphisms
$$
\mathcal{P}^d(\C)\cong \C^*\times\P^d(\C)
\ \mbox{ and }\ 
\mathcal{P}oly^{d,m}_n(\C)\cong
\T^m \times \Po^{d,m}_n,
$$
where we set $\T^m=(\C^*)^m$.
\qed
%%%
\end{remark}
%%(End of Remark 5.10)%%
\par\vspace{2mm}\par

Next consider the scanning map for algebraic maps.

%%(Definition 5.11)%%
\begin{definition}
%%%%%%%%%%%%%%%%%%
(i)
Let $U=D^2\setminus S^1=\{x\in\C :\vert x\vert <1\}$ and
define {\it the scanning map}
%%%%(5.8)%%
\begin{equation}\label{equ: scan2}
\mbox{sc}^{d,m}_n:\mathcal{P}oly^{d,m}_n\to \Map(\C, \mathcal{P}oly^{d,m}_n(U))
\end{equation}
%%%%
for  $\mathcal{P}oly^{d,m}_n$
by
$$
\mbox{sc}^{d,m}_n(f_1(z),\cdots ,f_m(z))(w)=
(f_1\vert U_w,\cdots ,f_m\vert U_w)
$$
for $((f_1(z),\cdots ,f_m(z),w)\in \mathcal{P}oly^{d,m}_n\times \C$,
where
 we also use the canonical
identification $U\cong U_w$
as in the definition of the earlier scanning map.
%%%%
\par\vspace{1mm}\par
(ii)
Let
%$p:\mathcal{P}oly^{d,m}_n\to\Po^{d,m}_n$ and
$q:\mathcal{P}oly^m_n(\C)\to \po^{m}_n(D^2,S^1)$
%%%
denote the map given by assigning to an $m$-tuples of polynomials their
corresponding roots in $U$. 
\qed
%%%%%%%%%%%%%%%%%%%%%%%%%%%
\end{definition}
%%(End of definition 5.11)%%%

%%%(Lemma 5.12)%%%%
\begin{lemma}\label{lmm: quasi-fibration}
%%%%%%%%%%%%%%%%%%
The map $q:\mathcal{P}oly^m_n(\C)\to \po^{m}_n(D^2,S^1)$ is a quasifibration with fibre $\T^m$.
%%%
\end{lemma}
%%%%%%%
\begin{proof}
%%%%%%
This may be proved by using the well-known criterion of Dold-Thom
as in the proof of \cite[Lemma 3.3]{Se}.
In fact, we can prove this by using the induction on the number $m$.
The case $m=1$ is proved in \cite[page 101]{GKY2}.
Now assume that the case $m-1$ is true and filter the base space
$\po^{d,m}_n(D^2,S^1)$ by the points of $U$ in the
first coordinate.
Note that $q$ is a trivial fibration over each
the successive difference of the filtration.
Then the Dold-Thom "attaching map"
has the effect of multiplying polynomials with no root in $U$ by a
fixed polynomial $z-\alpha$, where $\alpha\in \C\setminus U$.
Since $\alpha$ may be moved continuously to $1$, we can show the assertion
in the same way as the case $m=1$.
%%%
\end{proof}
%%(End of proof of Lemma 5.12)%%%%%

%%(Definition 5.13)%%
\begin{definition}
%%%%%%%%%%%%%%%%%%%%%
(i)
Let
$ev_0:\mathcal{P}oly^m_n(U)\to \C^{mn}\setminus \{{\bf 0}\}$
denote the evaluation map at $z=0$ given by
%%%%%
\begin{eqnarray*}
ev_0(f_1(z),\cdots ,f_m(z))
&=&
(\textit{\textbf{f}}_1(0),\textit{\textbf{f}}_2(0)\cdots ,
\textit{\textbf{f}}_m(0))
\end{eqnarray*}
for $(f_1(z),\cdots ,f_m(z))\in \mathcal{P}oly^{d,m}_n(U)$.
\par
(ii) Let $G$ be a group and $X$ a $G$-space.
Then we denote by $X//G$ the homotopy quotient of $X$ by $G$,
$X//G=EG\times_{G}X$.
\qed
%%%
\end{definition}
%%(End of definition 5.13)%%

%%(Remark 5.14)%%%
\begin{remark}
%%%
Let $\T^m=(\C^*)^m$ and
consider the natural $\T^m$-actions on the spaces
$\mathcal{P}oly^m_n(U)$ 
and $\C^{mn}\setminus \{{\bf 0}\}$ 
given by the usual coordinate-wise multiplications.
%$$
%\begin{cases}
%(g_1,\cdots ,g_m)\cdot (f_1(z),\cdots ,f_m(z))
%&=
%(g_1f_1(z),\cdots ,g_mf_m(z))
%\\
%(g_1,\cdots ,g_m)\cdot
%(\textit{\textbf{x}}_1,\cdots ,\textit{\textbf{x}}_m)
%&=
%(g_1\textit{\textbf{x}}_1,\cdots ,g_m\textit{\textbf{x}}_m)
%\qquad (\textit{\textbf{x}}_i\in\C^n)
%\end{cases}
%$$
%for $(g_1,\cdots ,g_m)\in \T^m$, where
%$(f_1(z),\cdots ,f_m(z))$ is an element of
% $\mathcal{P}oly^{d,m}_n$ or of
%$\mathcal{P}oly^m_n(U)$.
%%
Then it is easy to see that $ev_0$ is a $\T^m$-equivariant map.
\end{remark}
%%(End of Remark 5.14)%%%

%%(Lemma 5.15)%%
\begin{lemma}\label{lmm: evaluation0}
%%%%%%%%%%%%%%%
The map
$ev_0:\mathcal{P}oly^m_n(U)\to \C^{mn}\setminus \{{\bf 0}\}$
is a  homotopy equivalence.
\end{lemma}
%%%%%%%%%%%%%%
\begin{proof}
%%%%
If $m=1$, this is proved in \cite[Theorem 2.4 (page 102-103)]{GKY2}.
By using the same method with \cite[Prop. 1]{ha}, we can show the assertion for $m\geq 2$.
%%%
\end{proof}
%%%%%%%%%%%%
%\par\vspace{3mm}\par

%%%
Now we can prove Theorem \ref{thm: natural map}.

%%(Proof of Theorem 4.6)%%
\begin{proof}[Proof of Theorem \ref{thm: natural map}]
%%%%%%%%%%%%
Note that $\T^m$ does not act on 
$(\C^*)^{mn}\setminus \{{\bf 0}\}$  freely.
Hence,
by using its homotopy quotient, we have the  commutative diagram
%%%%%%%
%%%(Basic  commutative diagram)%%%%
{\small 
$$
\begin{CD}
%%%%%%%
\mathcal{P}oly^{d,m}_n 
@>\mbox{sc}^{d,m}_n>>
\Map (\C,\mathcal{P}oly^m_n(U))
@>ev_0>\simeq> 
\Map (\C, \C^{mn}\setminus\{{\bf 0}\})
%%%
\\
%%%
@V{q_1}VV @V{q_2}VV @V{q_3}VV
\\
%%%
\mathcal{P}oly^{d,m}_n/\T^m 
@>>> 
\Map (\C,\mathcal{P}oly^m_n(U)/\T^m)
@>\widetilde{ev}_0>\simeq> 
\Map (\C, (\C^{mn}\setminus \{{\bf 0}\})//\T^{m})
\\
@V{\cong}VV @V{q^{\p}}V{\simeq}V @.
%%%
\\
%%%
\Po^{d,m}_n @>sc^{d,m}_n>> \Map (\C, \po^m_n(D^2,S^1)) @.
%\Map (\C, S^{2mn-1})
\end{CD}
$$
}
\newline
where the vertical maps $q_i$ $(i=1,2,3)$
are induced maps from 
the corresponding group actions, and 
 $q^{\p}$ is induced map from the map $q$.
%%%
\par
Note
that the map $q^{\p}$ is a  homotopy equivalence
by  Lemma \ref{lmm: quasi-fibration}.
Since the map $ev_0$ is $\T^m$-equivariant and 
a homotopy equivalence by Lemma \ref{lmm: evaluation0},
the map $\widetilde{ev}_0$ is a homotopy equivalence.
\par\vspace{1mm}\par
%%%
Now consider the map $\gamma$ given by the second row of the above diagram
after imposing the base point condition at $\infty$.
If $d\to\infty$, $sc^{d,m}_n$ is a homology equivalence
by Theorem \ref{thm: scanning map}. So
the map $\gamma$ %(given by the second row of the above diagram)
is a homotopy equivalence if $d\to\infty$.
However,  since the map $q_3$ induces a homotopy equivalence
$\Omega^2S^{2mn-1}\simeq
\Omega^2_0((\C^{mn}\setminus \{{\bf 0}\})//\T^{m})$
(after imposing the base point condition at $\infty$),
this map coincides the map $i^{\infty,m}_n$
(if $d\to\infty)$
up to homotopy equivalence.
Hence, $i^{\infty,m}_n$ is a homology equivalence.
%%%%
\end{proof}
%%%(End of proof of Theorem 4.6)%%%%%

Next, we shall prove the unstable results
(Theorem \ref{thm: I} and Corollary \ref{thm: II}).
For this purpose, remark the following two results.

%%(Lemma 5.16)%%%
\begin{lemma}\label{lmm: abelian}
%%%%%%%%%%%%%%%
If $m\geq 2$ and $n\geq 2$, the space $\Po^{d,m}_n$ is 
simply connected.
%%%%%%%%%%%%%%
\end{lemma}
%%%%%%%%%%%%%%
\begin{proof}
%%%%%%%%%%%%%
Assume that $m\geq 2$ and $n\geq 2$.
Then 
by using the braid representation as in \cite[\S 5. Appendix]{GKY1},
any different kinds of strings can  pass through and
one can show that 
$\pi_1(\Po^{d,m}_n)$ is an abelian group.
Hence, there is an isomorphism
$\pi_1(\Po^{d,m}_n)\cong H_1(\Po^{d,m}_n,\Z)$.
Now consider the spectral sequence (\ref{SS}).
Then it follows from Corollary  \ref{crl: Er} that
$H_k(\Po^{d,m}_n,\Z)=0$
for any $1\leq k\leq 2mn-5$.
Thus, the space $\Po^{d,m}_n$ is 
simply connected.
%%%%%%%
\end{proof}
%%%%%%%%%%%%(End of proof of Lemma 6.1)%%%%

%%(Theorem 5.17)%%
\begin{theorem}\label{crl: stabilization}
%%%%%%%
If $m\geq 2$ and $n\geq 2$, the stabilization map
$$
s^{d,m}_n:\Po^{d,m}_n\to \Po^{d+1,m}_n
$$
is a homotopy equivalence if
$\lfloor \frac{d}{n}\rfloor =\lfloor \frac{d+1}{n}\rfloor$
and a homotopy equivalence through dimension $D(d;m,n)$
otherwise.
\end{theorem}
%%%
\begin{proof}
%%%%
This follows from Theorem \ref{thm: stab1}
and Lemma \ref{lmm: abelian}.
%%%
\end{proof}
%%%(End of proof of Theorem 5.17)%%%

Now we can complete the proof of Theorem \ref{thm: I}
and Corollary \ref{thm: II}.

%%Proof of Theorem 1.8)%%
\begin{proof}[Proof of Theorem \ref{thm: I}]
%%%%%%%%%%%%%%%%%%%%%%%%
Since the case $(m,n)=(3,1)$ and the case $(m,n)=(1,3)$ were
already proved in Theorem \ref{thm: KY7} and Theorem \ref{thm: KY6},
without loss of generalities, we may assume that $m\geq $ and $n\geq 2$.
It follows from Theorem \ref{thm: stab1} and Theorem
\ref{thm: natural map} that
the map $i^{d,m}_n:\Po^{d,m}_n\to \Omega^2_0\CP^{mn-1}$
is a homology equivalence through dimension $D(d;m,n)$.
However, since $\Po^{d,m}_n$ and $\Omega^2_0\CP^{mn-1}\simeq
\Omega^2S^{2mn-1}$ are simply connected,
the map $i^{d,m}_n$ is indeed a homotopy equivalence through dimension
$D(d;m,n)$.
%%%
\end{proof}
%%%(End of proof fo Theorem 1.8)%%%

%%%(Proof of Corollary 1.9)%%%
\begin{proof}[Proof of Corollary \ref{thm: II}]
%%%%%%%%%%%%%%%%%%%%%%%%%%%%%
By the same reason as the proof of Theorem \ref{thm: I},
it suffices to prove the assertion when $m\geq 2$ and $n\geq 2$.
By the diagram (\ref{CD: stab0}), it is easy to see that
the following diagram is commutative (up to homotopy).
$$
\begin{CD}
\Po^{d,m}_n @>i^{d,m}_n>> \Omega^2_d\CP^{mn-1}
%@>>\simeq> \Omega^2_0\CP^{mn-1}
\\
@V{j^{d,m}_n}VV  \Vert @.
\\
\Hol_d^*(S^2,\CP^{mn-1}) @>i_d^{\p}>>
\Omega^2_d\CP^{mn-1}
%@>>\simeq> \Omega^2_0\CP^{mn-1}
\end{CD}
$$
It follows from Theorem \ref{thm: KY6}  that $i_d^{\p}$ is a homotopy equivalence through dimension
$(2mn-3)(d+1)-1$.
Moreover, by Theorem \ref{thm: I} we know that
$i^{d,m}_n$ is a homotopy equivalence through dimension $D(d;m,n)$.
Since $D(d;m,n)<(2mn-3)(d+1)-1$,
it follows from  above diagram that the map $j^{d,m}_n$ is a homotopy equivalence through dimension $D(d;m,n)$.
%%%
\end{proof}
%%(End of proof of Corollary 1.9)%%%%%

%%%(SECTION 6)%%%%
\section{$C_2$-structures}\label{section: III}
%%%%%%%%%%%%%%%%%

In this section we shall prove Theorem \ref{thm: III}
and Theorem \ref{thm: IV}.

%%%(Definition 6.1)%%
\begin{definition}
%%%%%%%%%%%%%%%%%%%
Let 
$J=(0,1)$ be an open unit disk in $\R$.
Recall that {\it a little $2$-cube} $c$ means an affine embedding
$c:J^2\to J^2$ with parallel axes.
\par
(i)
For each integer $j\geq 1$, let $C_2(j)$ denote the space consisting of all
$j$-tuples $(c_1,\cdots ,c_j)$ of little $2$-cubes such that
$c_k(J^2)\cap c_i(J^2)=\emptyset$ if $k\not= i$.
\par
%%(ii)%%
(ii)
Let $\alpha_0:\C\stackrel{\cong}{\longrightarrow}J^2$.
If $S_j$ denotes the symmetric group of
$j$-letters,
it is easy to see that it acts on the spaces
$C_2(j)$ and $\Po^{d.m}_n$ by the permutation of coordinates
in a natural way.
Then we identify $\C =\R^2$ and
define {\it the structure map}
%%()%%
%\begin{equation*}
%%%
$
\mathcal{I}_j:C_2(j)\times_{S_j}(\Po^{d,m}_n)^j\to
\Po^{dj,m}_n
$
%%
%\end{equation*}
%%
by
%%(6.1)%%
\begin{equation}
%%%
\mathcal{I}_j((c_1,\cdots ,c_j),(f_1,\cdots ,f_j))=
(\prod_{k=1}^jc_k(f_{1;k}(z)),\cdots ,\prod_{k=1}^jc_k(f_{m;k}(z)))
\end{equation}
%%%
for $(c_1,\cdots ,c_j)\in C_2(j)$ and
$f_k=(f_{i;k}(z),\cdots ,f_{m;k}(z))\in \Po^{d,m}_n$ 
($1\leq k\leq j$), where 
we set
%%
%%%%(6.2)%%
\begin{equation}
c_k(f(z))=\prod_{i=1}^d\big(z-c_k\circ\alpha_0(a_i)\big)
\quad
\mbox{if }f(z)=\prod_{i=1}^d(z-a_i)\in \P^d(\C).
\end{equation}
%%%
\par
%%(iii)%%
(iii)
Similarly,
let $c_*=(c_{1},c_2)\in C_2(2)$
be any fixed element and define {\it the loop product}
$
\mu_{d_1,d_2} :\Po^{d_1,m}_n\times \Po^{d_2,m}_n\to \Po^{d_1+d_2,m}_n
$
by
%%%(6.3)%%
\begin{equation}
%%%%%%%%%%
\mu_{d_1,d_2} (f,g)=
(c_1(f_1(z))c_2(g_1(z)),\cdots
,c_1(f_m(z))c_2(g_m(z)))
\end{equation}
%%%%
for $(f,g)=\big((f_1(z),\cdots ,f_m(z)),(g_1(z),\cdots ,g_m(z))\big)
\in \Po^{d_1,m}_n\times \Po^{d_2,m}_n$.
%%%
%%%
\end{definition}
%%%%(End of Definition 6.1)%%%

%%%(Remark 6.2)%%%
\begin{remark}
%%%
(i)
It is easy to see  that $\mu_{d_1,d_2} (f,g)=\mathcal{I}_2(c_*;f,g)$ if $d_1=d_2=d$.
\par
(ii)
Let $\Po^{0,m}_n=\{*_0\}$ and let 
$\Po^{m}_n$ denote the disjoint union
%%(6.4)%%
\begin{equation}
%%%%%
\mbox{Poly}^{m}_n=\coprod_{d\geq 0}\Po^{d,m}_n.
\end{equation}
%%%%%%
If we set $\mu_{d,0}(f,*_0)=\mu_{0,d}(*_0,f)=f$ for any
$f\in \Po^{d,m}_n$, it is easy to see that
$\mbox{Poly}^{m}_n$ is a homotopy associative H-space with unit
$*_0$, and
we can easily see that
$\{\mbox{Poly}^{m}_n,\mathcal{I}_j\}_{j\geq 1}$ is a $C_2$-operad space.
%%%
Thus,
by using the group completion Theorem and Theorem \ref{thm: I},
we see that
there is a homotopy equivalence
%%%(6.5)%%
\begin{equation}\label{eq: group completion}
\Omega B(\mbox{Poly}^{m}_n)\simeq
\Omega^2\CP^{mn-1}\simeq \Z \times \Omega^2S^{2mn-1}.
\quad
\qed
\end{equation}
\end{remark}
%%(End of Remark 6.2)%%%

%%%(Definition 6.3)%%
\begin{definition}
%%%
Let 
\begin{eqnarray*}
*:&&\Hol_{d_1}^*(S^2,\CP^{mn-1})\times
\Hol_{d_2}^*(S^2,\CP^{mn-1})\to \Hol_{d_1+d_2}^*(S^2,\CP^{mn-1})
\\
&&\mbox{and}
\\
\mathcal{I}:&&C_2(j)\times_{S_j}\Hol_{d}^*(S^2,\CP^{mn-1})^j
\to \Hol_{dj}^*(S^2,\CP^{mn-1})
\end{eqnarray*}
denote the loop product and the
$C_2$-structure map 
defined in
\cite[(4.8)]{BM} and 
\cite[(4.10)]{BM}, respectively.
\end{definition}
%%%(End of Definition 6.3)%%%

It is easy to see that the above definitions of the loop products and the structure maps
are completely analogous to those of \cite[(4.8), (4.10)]{BM}, and we have the following:

%%%(Lemma 6.4)%%
\begin{lemma}\label{lmm: C2}
%%%%%%%%%%%%%%
\begin{enumerate}
\item[$\I$]
The following two diagrams are homotopy commutative.
$$
\begin{CD}
\Po^{d_1,m}_n\times \Po^{d_2,m}_n 
@>{\mu_{d_1,d_2}}>> \Po^{d_1+d_2,m}_n
\\
@V{j^{d_1,m}_n\times j^{d_2,m}_n}VV @V{j^{d_1+d_2,m}_n}VV
\\
\Hol_{d_1}^*(S^2,\CP^{mn-1})
\times
\Hol_{d_2}^*(S^2,\CP^{mn-1})
@>*>> \Hol_{d_1+d_2}^*(S^2,\CP^{mn-1})
\end{CD}
$$
$$
\begin{CD}
C_2(j)\times_{S_j}(\Po^{d,m}_n)^j @>{\mathcal{I}_j}>> \Po^{dj,m}_n
\\
@V{1\times (j^{d,m}_n)^j}VV @V{j^{dj,m}_n}VV
\\
C_j(2)\times_{S_j} \Hol_d^*(S^2,\CP^{mn-1})^j 
@>\mathcal{I}>>
\Hol_{dj}^*(S^2,\CP^{mn-1})
\end{CD}
$$
%%(ii)%%
\item[$\II$]
%%%
The map
$$
\dis\coprod_{d\geq 0}i^{d,m}_n:\mbox{\rm Poly}^m_n=\coprod_{d\geq 0}\Po^{d,m}_n \to \coprod_{d\in\Z}\Omega^2_d\CP^{mn-1}=\Omega^2\CP^{mn-1}
$$
is a $C_2$-map up to homotopy equivalence.
\end{enumerate}
%%%
\end{lemma}
%%%(End of Lemma 6.4)%%
\begin{proof}
%%%
This can be proved by an analogous way as
given in \cite[Theorem 4.10]{BM}.
\end{proof}
%%(End of proof of Lemma 6.4)%%

%%(Lemma 6.5)%%
\begin{lemma}\label{lmm: Pon}
%%%%%%
There is a homotopy equivalence
$\Po^{n,m}_n\simeq S^{2mn-3}$.
\end{lemma}
%%%%
\begin{proof}
%%%
This easily follows from the homeomorphism
(\ref{eq: Poly}).
\end{proof}
%%(End of proof of Lemma 6.5)%%

%%(Lemma 6.6)%%
\begin{lemma}[\cite{CMM}, \cite{CMT}, \cite{Sn}]
\label{lmm: Snaith}
%%%%%%%%%%%%%
\begin{enumerate}
\item[$\I$]
If $X$ is a connected based CW complex,
there is a stable homotopy equivalence
$\dis
\Omega^2\Sigma^2X\simeq_s
\bigvee_{k=1}^{\infty}D_k(X),
$
where $D_k(X)$ denotes the space
$D_k(X)=F(\C,k)_+\wedge_{S_k}(\bigwedge^k X)$ as in (\ref{equ: DkX}).
\item[$\II$]
For integers $k\geq 1$ and $N\geq 2$,
there is a homotopy equivalence
$D_k(S^{2N-1})\simeq \Sigma^{2(N-1)k}D_k$,
where  $D_k=D_k(S^1)$ as in (\ref{equ: DkS}).
%%(iii)%%
\item[$\III$]
The canonical projection
$p_{k,N}:F(\C,k)\times_{S_k}(S^{2N-1})^k\to D_k(S^{2N-1})$
has the stable section
$e_{k,N}: D_k(S^{2N-1})\to F(\C,k)\times_{S_k}(S^{2N-1})^k$.
\qed
\end{enumerate}
%%%
\end{lemma}
%%%%%%%%%%

%%(Definition 6.7)%%
\begin{definition}\label{dfn: Phi, Psi}
%%%%%%
(i)
For each $1\leq k<d$,
let $s_{k,d}:\Po^{kn,m}_n\to \Po^{dn,m}_n$ denote
the composite of stabilization maps
%%(6.6)%%
\begin{equation}
s_{k,d}:
\Po^{kn,m}_n\stackrel{}{\longrightarrow}
\Po^{kn+1,m}_n\stackrel{}{\longrightarrow}
\cdots \stackrel{}{\longrightarrow}
\Po^{dn-1,m}_n\stackrel{}{\longrightarrow}
\Po^{dn,m}_n.
\end{equation}
We denote by 
$\Po^{dn,m}_n/\Po^{(d-1)n,m}_n$ the mapping cone of the map
$s_{k,d}$.
\par
%%(ii)%%
(ii) Let
%%
%%%(6.7)%%%%
\begin{equation}\label{eq: ed def}
%%%%%%%%%%%%
e_d:\Sigma^{2(mn-2)d}D_d\to F(\C,d)\times_{S_d}(S^{2mn-3})^d
%%%%%%%%
\end{equation}
%%%%
denote the stable map defined by the composite of maps
$$
\begin{CD}
e_d:\Sigma^{2(mn-2)d}D_d\simeq D_d(S^{2mn-3})@>e_{d,mn-1}>>
F(\C,d)\times_{S_d}(S^{2mn-3})^d.
\end{CD}
$$
It is easy to see that $e_d$ is a stable section of the projection
$F(\C,d)\times_{S_d}(S^{2mn-3})^d\to D_d(S^{2mn-3})\simeq \Sigma^{2(mn-2)d}D_d.$
%%%
%%(iii)%%
\par
(iii)
Since
there is a $S_d$-equivariant homotopy equivalence
$C_2(d)\times_{S_d}F(\C,d)$, one can
define the stable map
%%%%%%%%%%%%%%
$
\Psi_d:\Sigma^{2(mn-2)d}D_d\to \Po^{dn,m}_n/\Po^{(d-1)n,m}_n
%%%%%%%%%%%%%%%
$
by
%%
%by the composite of maps given by
%%(6.8)%%
\begin{equation}
\Psi_d=\tilde{p}_d\circ \mathcal{I}_d^{\p}\circ e_d,
\end{equation}
where
$\tilde{p}_d:\Po^{dn,m}_n\to \Po^{dn,m}_n/\Po^{(d-1)n,m}_n$
is the natural projection and the map
$\mathcal{I}_d^{\p}$ denotes the map defined by the composite of maps
%%(6.9)%%
\begin{equation}
%%%%%
\mathcal{I}_d^{\p}:
F(\C,d)\times_{S_d}(S^{2mn-3})^d\simeq
C_2(d)\times_{S_d}(\Po^{n,m}_n)^d
\stackrel{\mathcal{I}_d}{\longrightarrow}
\Po^{dn,m}_n.
\end{equation}
%%%%%%%%
Note that the following diagram
is commutative.
%%%%%%%%%%
%%()%%
\begin{equation*}\label{CD: Psi}
\begin{CD}
\Sigma^{2(mn-2)d}D_d
\simeq_s
D_d(S^{2mn-3})@>e_d>>
F(\C ,d)\times_{S_d}(S^{2mn-3})^d
\\
@V{\Psi_d}VV @V{\mathcal{I}_d^{\p}}VV
\\
\Po^{dn,m}_n/\Po^{(d-1)n,m}_n @<\tilde{p}_d<< \Po^{dn,m}_n
%%%
\end{CD}
\end{equation*}
%%%%
%%%%%%%%
%%
Similarly, define the stable map
%%%
%%()%%
%\begin{equation*}
%%%%
$
\Phi_d:\bigvee_{k=1}^d\Sigma^{2(mn-2)k}D_k \to \Po^{dn,m}_n
$
%%%%%%
%\end{equation*}
%%%
by 
%%(6.10)%%
\begin{equation}
%%%%
\Phi_d=(\vee s_{k,d})\circ (\vee \mathcal{I}_d^{\p})\circ (\vee e_k),
%%%
\end{equation}
%%%%%
Thus the following diagram is commutative:
%%%
%%%()%%%
\begin{equation*}
\begin{CD}
\dis 
\bigvee_{k=1}^d\Sigma^{2(mn-2)k}D_k
@>>\simeq_s>
\dis
\bigvee_{k=1}^dD_k(S^{2mn-3})
@>{\vee e_k}>>
\dis
\bigvee_{k=1}^d F(\C ,k)\times_{S_k}(S^{2mn-3})^k
\\
@V{\Phi_d}VV @. @V{\vee \mathcal{I}_k^{\p}}VV
\\
\Po^{dn,m}_n 
@>=>> \Po^{dn,m}_n
@<{\vee s_{k,d}}<< \dis
\bigvee_{k=1}^d \Po^{kn,m}_n
%%%
\end{CD}
\end{equation*}
%%%%%
\par
%%(iv)%%%
(iv)
For a connected space $X$, let $J_2(X)$ denote 
{\it the May-Milgram model} for  $\Omega^2\Sigma^2X$ \cite{May}
%%%%%%%%%%
%%%%%%%%%
%%(6.11)%%
\begin{equation}
J_2(X)=\big(\coprod_{k= 1}^{\infty}F(\C,k)\times_{S_k}X^k\big)/\sim ,
\end{equation}
%%%
where $\sim$ denotes the well-known equivalence relation.
For each integer $d\geq 1$, let $J_2(X)_d\subset J_2(X)$
denote {\it the $d$-th stage filtration of the May-Milgram model} for
$\Omega^2\Sigma^2X$ defined by
%%%%%%
%%(6.12)%%%%
\begin{equation}\label{eq: d-th MM-model}
%%%%%%%%%%%%
J_2(X)_d=\big(\coprod_{k=1}^dF(\C,k)\times_{S_k}X^k\big)/\sim.
\quad
\qed
\end{equation}
%%%%%%%%%%%%%
%%%%%%%%%%%%%%%%%%%
\end{definition}
%%(End of Definition 6.7)%%

%%(Theorem 6.8)%%
\begin{theorem}\label{thm: V}
%%%%%%
The map
$\Psi_d:\Sigma^{2(mn-2)d}D_d
\stackrel{\simeq_s}{\longrightarrow}
 \Po^{dn,m}_n/\Po^{(d-1)n,m}_n$
 is a stable homotopy equivalence.
%%%
\end{theorem}
%%%%%%

The proof of Theorem \ref{thm: V} is postponed to \S \ref{section: Proof V}, and
we prove the following result.

%%(Theorem 6.9)%%%
\begin{theorem}\label{thm: VI}
%%%%%%%%%%%%%%%%%
The map
$\dis \Phi_d:\bigvee_{k=1}^d\Sigma^{2(mn-2)k}D_k 
\stackrel{\simeq_s}{\longrightarrow} \Po^{dn,m}_n$
is a stable homotopy equivalence.
\end{theorem}
%%%%
%%(Proof of Theorem 6.9)%%
\begin{proof}
%%%%%%
We proceed by the induction on $d$.
If $d=1$, since there is a homotopy equivalence
$D_1\simeq S^1$, the assertion follows from Lemma \ref{lmm: Pon}.
Assume that the result holds for the case $d-1$, i.e.
the map $\Phi_{d-1}$ is a stable homotopy equivalence.
Note that the following diagram is commutative.
$$
\begin{CD}
\bigvee_{k=1}^d\Sigma^{2(mn-2)k}D_k @>\Phi_d>>
\Po^{dn,m}_n
\\
%%%
\Vert @. @A{s_{d-1,d} \vee 1}AA 
\\
%%%
\big(\bigvee_{k=1}^{d-1}\Sigma^{2(mn-2)k}D_k\big)
\vee \Sigma^{2(mn-2)d}D_d
@>\Phi_{d-1}\vee \mathcal{I}_d^{\p}\circ e_d>>
\Po^{(d-1)n,m}_n\vee
\Po^{dn,m}_n
\end{CD}
$$
Thus, we can easily obtain the following 
homotopy commutative diagram
$$
\begin{CD}
\bigvee_{k=1}^{d-1}\Sigma^{2(mn-2)k}D_k
@>>\subset>
\bigvee_{k=1}^{d-1}\Sigma^{2(mn-2)k}D_k
@>>>
\Sigma^{2(mn-2)d}D_d
\\
@V{\Phi_{d-1}}V{\simeq_s}V @V{\Phi_d}VV @V{\Psi_d}V{\simeq_s}V
\\
\Po^{(d-1)n,m}_n
@>{s_{d-1,d}}>>
\Po^{dn,m}_n
@>\tilde{p}_d>>
\Po^{dn,m}_n/
\Po^{(d-1)n,m}_n
\end{CD}
$$
where the horizontal sequences are cofibration sequences.
Since $\Phi_{d-1}$ and $\Psi_d$ are stable homotopy equivalences,
the map $\Phi_d$ is so.
%%%%%%%%%%%%%%%%%%%%
\end{proof}
%%%(End of proof of Theorem 6.9)%%%%
%%

Now it is ready to prove Theorem \ref{thm: III}.

%%%(Proof of Theorem 1.10)%%
\begin{proof}[Proof of Theorem \ref{thm: III}]
%%%%%%%%%
Since the case $m=1$ or the case $n=1$ were obtained by
Theorem \ref{thm: CCMM} and Theorem \ref{thm: GKY4},
assume that $m\geq 2$ and $n\geq 2$.
Then the assertion easily follows from Theorem \ref{crl: stabilization}
and Theorem \ref{thm: VI}.
%%%
\end{proof}
%%(End of proof of Theorem 1.10)%%

Next we shall prove Theorem \ref{thm: IV}.
%%%(Definition 6.10)%%
\begin{definition}
%%%%%%%%%%%%%%%%%
It follows from Lemma \ref{lmm: C2} and
\cite[Theorem 4.14, Theorem 4.16]{BM}
that $C_2$-structure of $\mbox{Pol}^m_n=\coprod_{d\geq 0}\Po^{d,m}_n$ and that of 
$J_2(S^{2mn-3})$
induced from the double loop product are compatible.
So the structure maps $\mathcal{I}_d$'s induce a map
%%
%%(6.13)%%
\begin{equation}
%%%%%%%%%
\epsilon_d:J_2(S^{2mn-3})_d\to \Po^{dn,m}_n
\end{equation}
%%%
such that the following diagram is homotopy commutative:
%%%(6.14)%%
\begin{equation}
\begin{CD}
\bigvee_{k=1}^dF(\C.k)\times_{S_k}(S^{2mn-3})^k 
@>\vee\mathcal{I}_k^{\p}>>
\bigvee_{k=1}^d\Po^{kn,m}_n
\\
@V{\vee q_k}VV @V{\vee s_{k,d}}VV
\\
J_2(S^{2mn-3})_d
@>\ep_d>> \Po^{dn,m}_n
\end{CD}
\end{equation}
where
%%%(6.15)%%
\begin{equation}
%%%
q_k:F(\C,k)\times_{S_k}(S^{2mn-3})^k\to
J_2(S^{2mn-3})_d
\qquad (1\leq k\leq d)
%%%
\end{equation}
%%%%%%
denotes the natural projection.
\qed
%%%%
%%%
\end{definition}
%%%%%(End of Definition 6.10)%%

%%(Theorem 6.11)%%
\begin{theorem}\label{thm: VII}
%%%%%%%%%%%%%%%%%
If $m\geq 2$ and $n\geq 2$,
the map
$\epsilon_d:J_2(S^{2mn-3})_d
\stackrel{\simeq}{\longrightarrow}\Po^{dn,m}_n$
is a homotopy equivalence.
%%%
\end{theorem}
%%%%%%%%%%%%
%%(Proof of Theorem 6.11)%%
\begin{proof}
%%%%%%%%%%%%%%%%%%%
%%%
%%%
Since stable maps $\{e_k\}$ are stable sections of the
Snaith splittings (by Lemma \ref{lmm: Snaith}),
the map
%%()%%
%\begin{equation}\label{eq: the map J}
%%%%%%%%
$\dis (\vee q_k)\circ (\vee e_k):\bigvee_{k=1}^d\Sigma^{2(mn-2)k}D_k
\stackrel{\simeq_s}{\longrightarrow}J_2(S^{2mn-3})_d$
%\end{equation}
%%%
is a stable homotopy equivalence.
Now, it is easy to see that the following diagram is stable homotopy
commutative:
{\small
$$
\begin{CD}
%%%%%%%
\dis\bigvee_{k=1}^d\Sigma^{2(mn-2)k}D_k
@>{\vee e_k}>> 
\dis\bigvee_{k=1}^dF(\C,k)\times_{S_k}(S^{2mn-3})^k
@>{\vee \mathcal{I}_k^{\p}}>>
\dis\bigvee_{k=1}^d\Po^{kn,m}_n
\\
\Vert @. @V{\vee q_k}VV @V{\vee s_{k,d}}VV
\\
\dis
\bigvee_{k=1}^d\Sigma^{2(mn-2)k}D_k
@>(\vee q_k)\circ (\vee e_k)>{\simeq_s}>
J_2(S^{2mn-3})_d
@>{\ep_d}>>
\Po^{dn,m}_n
\end{CD}
$$}
%%%%%%
%%%%%
Since $\Phi_d=(\vee s_{k,d})\circ (\vee \mathcal{I}_k^{\p})
\circ (\vee e_k)$ and it is a stable homotopy equivalence
(by Theorem \ref{thm: VI}), the map
$\ep_d$ is so.
Thus, the map $\ep_d$ is a homology equivalence.
Since 
$J_2(S^{2mn-3})_d$ and $\Po^{dn,m}_n$ are simply connected
(by Lemma \ref{lmm: abelian}),
the map $\ep_d$ is a homotopy equivalence. 
%%%%%%%%%%%%%%%%%%%%%%%%%%%%%
\end{proof}
%%%(End of proof of Theorem 6.11)%%

Now we can give the proof of 
Theorem \ref{thm: IV}.

%%%(Proof of Theorem 1.11)%%
\begin{proof}[Proof of Theorem \ref{thm: IV}]
%%%%%%%%%
%If $mn= 3$, then $(m,n)=(3,1)$ or $(1,3)$.
If $(m,n)=(3,1)$, then
$\Po^{d,m}_n=\Hol_d^*(S^2,\CP^2)=\Hol_{\lfloor\frac{d}{n}\rfloor}^*(S^2,\CP^{mn-1})$
and the assertion  holds.
If $(m,n)=(1,3)$, it was already proved by
Theorem \ref{thm: GKY4} and
 assume that
$m\geq 2$ and $n\geq 2$.
%%%
%%%%%
It follows from Theorem \ref{crl: stabilization}
that
it suffices to prove that
there is a homotopy equivalence
%%%%%%%%
%%(6.16)%%
\begin{equation}\label{eq: equiv}
%%%%%%%%
\Po^{dn,m}_n\simeq \Hol_d^*(S^2,\CP^{mn-1}).
\end{equation}
By Theorem \ref{thm: VII},
$
\ep_d:J_2(S^{2mn-3})_d\stackrel{\simeq}{\longrightarrow}\Po^{dn,m}_n
$
is a homotopy equivalence.
On the other hand, it follows from (ii) of Theorem \ref{thm: CCMM} that there is a
homotopy equivalence
$
J_2(S^{2mn-3})_d\simeq \Hol_d^*(S^2,\CP^{mn-1}).
$
Hence, there is a homotopy equivalence (\ref{eq: equiv})
and the assertion is obtained.
%%%%%%%%%
\end{proof}
%%(End of proof of Theorem 1.11)%%

%%%(SECTION 7)%%
\section{The space $\Po^{dn,m}_n/\Po^{(d-1)n,m}_n$}\label{section: Proof V}
%%%%%%%%%%%%%%%%

In this section we give the proof of Theorem \ref{thm: V}.
Since the case $m=1$ and the case $n=1$ were already proved
in \cite[Theorem 1, Theorem 15]{CCMM2} and \cite[Theorem 2.9]{GKY4},
in this section we always assume that $m\geq 2$ and $n\geq 2$.
Let
%%(8.1)%%
\begin{equation}
%%%% 
\iota_d:
%\lim_{t\to\infty} s_{d,t}
\Po^{dn,m}
\to 
%\lim_{t\to\infty}\Po^{tn,m}_n=
\Po^{\infty ,m}_n
\end{equation}
%%%%%
denote the natural inclusion map induced from the stabilization maps.

%%%(Theorem 7.1)%%
\begin{theorem}\label{thm: VIII}
%%%%%%%
%%
The stable map
$$
\dis
(\vee \iota_d)\circ (\vee \mathcal{I}_d^{\p})
\circ (\vee e_d):
\bigvee_{d=1}^{\infty}\Sigma^{2(mn-2)d}D_d
%\to
%\bigvee_{d=1}^{\infty}\Po^{dn,m}_n
\stackrel{\simeq_s}{\longrightarrow}
\Po^{\infty ,m}_n
$$
is a stable homotopy equivalence.
\end{theorem}
%%%%(Proof of Theorem 8.1)%%
\begin{proof}
%%%%%%%%%%%
By using (ii) of Lemma \ref{lmm: C2} we see that
$i^{\infty ,m}_n:\Po^{\infty ,m}_n\to \Omega^2S^{2mn-1}$ is
a $C_2$-map such that the following diagram is homotopy commutative
%%()%%
\begin{equation*}\label{CD: a}
%%%
\begin{CD}
J_2(\Po^{\infty ,m}_n)
@>J_2(i^{\infty ,m}_n)>\simeq> J_2(\Omega^2S^{2mn-1})
\\
@V{r_1}VV @V{r_2}VV
%%%
\\
\Po^{\infty ,m}_n @>i^{\infty ,m}_n>\simeq>\Omega^2S^{2mn-1}
\end{CD}
\end{equation*}
%%%%
where $r_1$ and $r_2$ are natural retraction maps.
Since $m\geq 2$ and $n\geq 2$, 
the two maps $i^{\infty ,m}_n$ and $J_2(i^{\infty ,m}_n)$ are
indeed homotopy equivalences (by Theorem \ref{thm: natural map}
and Lemma \ref{lmm: abelian}).
\par
Similarly, by using (ii) of Lemma \ref{lmm: C2} we have the following homotopy commutative diagram
%%%%()%% 
\begin{equation*}\label{CD: b}
%%%%%%%%%
\begin{CD}
\bigvee_{d=1}^{\infty}F(\C ,d)\times_{S_d}(S^{2mn-3})^d
@>{\vee p_d}>> J_2(S^{2mn-3}) @>J_2(\iota )>>
J_2(\Po^{\infty ,m}_n)
%%%%
\\
%%%
@V{\vee \mathcal{I}_d^{\p}}VV @. @V{r_1}VV
\\
\bigvee_{d=1}^{\infty}\Po^{dn,m}_n
@>{\vee \iota_d}>>\Po^{\infty,m}_n 
@>=>>\Po^{\infty,m}_n
\end{CD}
\end{equation*}
%%%%%
where
$\iota$ denotes the natural inclusion map
$\iota :S^{2mn-3}\simeq \Po^{n,m}_n\stackrel{\iota_1}{\longrightarrow}
\Po^{\infty ,m}_n.$
Now consider the composite of maps
$$
\begin{CD}
J_2(S^{2mn-3})
@>J_2(\iota)>>
J_2(\Po^{\infty,m}_n)
@>J_2(i^{\infty ,m}_n)>\simeq> J_2(\Omega^2S^{2mn-1})
@>r_2>> \Omega^2S^{2mn-1}
\end{CD}
$$
Since the map $i^{\infty,m}_n\circ \iota$ is the natural
inclusion of the bottom cell of the double suspension
$E^2:S^{2mn-3}\to \Omega^2\Sigma^2S^{2mn-3}=\Omega^2S^{2mn-1}$
(up to homotopy equivalence), the map
$r_2\circ J_2(i^{\infty ,m}_n)\circ J_2(\iota)$ is
homotopic to the natural homotopy equivalence
$J_2(S^{2mn-3})\simeq \Omega^2S^{2mn-1}$.
%%%
Thus the above two diagrams reduce to the following stable homotopy 
commutative diagram 
%%%%%%%%
\begin{equation*}
\begin{CD}
\bigvee_{d=1}^{\infty}\Sigma^{2(mn-2)d}D_d @.
\\
@V{\vee e_d}VV @.
\\
%%%(new)%%
\bigvee_{d=1}^{\infty}F(\C,d)\times_{S_d}(S^{2mn-3})^d
@>{\vee p_d}>>
J_2(S^{2mn-3}) \simeq
\Omega^2S^{2mn-1}
%\bigvee_{d=1}^{\infty}\Po^{dn,m}_n
%%%
\\
@V{\vee \mathcal{I}_d^{\p}}VV @A{i^{\infty ,m}_n}A{\simeq}A
\\
\bigvee_{d=1}^{\infty}\Po^{dn,m}_n
@>{\vee \iota_d}>>
\Po^{\infty,m}_n
\end{CD}
\end{equation*}
%%%
Since the stable maps $\{e_d\}$ are stable sections of the stable homotopy equivalence
$\Omega^2S^{2mn-1}\simeq_s\vee_{d=1}^{\infty}\Sigma^{2(mn-2)d}D_d$,
the map
$(\vee p_d)\circ (\vee e_d)$ is a stable homotopy equivalence
and
the map
$(\vee \iota_d) \circ  (\vee \mathcal{I}_d^{\p})\circ 
(\vee e_d)$ 
is so.
%%%
\end{proof}
%%(End of proof of Theorem 7.1)%%%%
%%
%%
%%%%(Remark 7.2)%%%%
\begin{remark}\label{remark: rem}
Since $\dis\lim_{d\to\infty}\Phi_d
=(\vee \iota_d) \circ  (\vee \mathcal{I}_d^{\p})\circ 
(\vee e_d)$, the
above result may be regarded as the stable version of 
Theorem \ref{thm: VI}.
\qed
%%%%
\end{remark}
%%%%(End of Remark 7.2)%%%%%

%%%(Lemma 7.3)%%
\begin{lemma}\label{lmm: X}
%%%%%%%%%%%%%%%%
%%%%
$\I$
The induced homomorphism
$$
(s_{d-1,d})_*:
H_*(\Po^{(d-1)n,m}_n,\Z)\to H_*(\Po^{dn,m}_n,\Z)
$$
is a monomorphism.
\par
$\II$
The induced homomorphism
$$
(\Psi_{d})_*: H_*(\Sigma^{2(mn-2)d}D_d,\F)\to
H_*(\Po^{dn,m}_n/\Po^{(d-1)n,m}_n,\F)
$$
is a monomorphism for $\F =\Q$ or $\Z/p$
$(p:$ any prime$)$.
%%
%%%%%%%%%
\end{lemma}
%%%%%(End of Lemma 7.3)%%

We postpone the proof of Lemma \ref{lmm: X} to \S \ref{section: transfer}
and we give the proof of Theorem \ref{thm: V}.

%%(Proof of Theorem 6.8)%%
\begin{proof}[Proof of Theorem \ref{thm: V}]
%%%%%%%%%%%%
Let $\F =\Q$ or $\Z/p$
$(p:$ any prime$)$.
%%%
By (i) of Lemma \ref{lmm: X},
there is an isomorphism of $\F$-vector spaces
$$
H_*(\Po^{\infty,m}_n,\F)\cong 
H_*(\bigvee_{d=1}^{\infty}\Po^{dn,m}_n/\Po^{(d-1)n,m}_n,\F).
$$
Hence, it follows from Lemma \ref{lmm: Snaith}, 
Theorem \ref{thm: natural map} and
Theorem \ref{thm: VIII} 
that there is an isomorphism of $\F$-vector spaces
$$
H_*(\bigvee_{d=1}^{\infty}\Sigma^{2(mn-2)d}D_d,\F)\cong
H_*(\bigvee_{d=1}^{\infty}\Po^{dn,m}_n/\Po^{(d-1)n,m}_n,\F).
$$
Thus the following equality holds for each $k\geq 1$:
$$
\dim_{\F}H_k(\bigvee_{d=1}^{\infty}\Sigma^{2(mn-2)d}D_d,\F)=
\dim_{\F}H_k(\bigvee_{d=1}^{\infty}\Po^{dn,m}_n/\Po^{(d-1)n,m}_n,\F)
<\infty
$$
However, by (ii) of Lemma \ref{lmm: X},
we see that the  homomorphism
$$
(\vee_d\Psi_{d})_*: 
H_*(\bigvee_{d=1}^{\infty}\Sigma^{2(mn-2)d}D_d,\F)\to
H_*(\bigvee_{d=1}^{\infty}\Po^{dn,m}_n/\Po^{(d-1)n,m}_n,\F)
$$
is injective.
%%%%
Therefore, the  homomorphism
$$
(\vee_d\Psi_{d})_*: 
H_*(\bigvee_{d=1}^{\infty}\Sigma^{2(mn-2)d}D_d,\F)
\stackrel{\cong}{\longrightarrow}
H_*(\bigvee_{d=1}^{\infty}\Po^{dn,m}_n/\Po^{(d-1)n,m}_n,\F)
$$
is indeed an isomorphism for
$\F =\Q$ or $\Z/p$
$(p:$ any prime$)$, and
from the universal coefficient Theorem,
the homomorphism
$$
(\vee_d\Psi_{d})_*: 
H_*(\bigvee_{d=1}^{\infty}\Sigma^{2(mn-2)d}D_d,\Z)
\stackrel{\cong}{\longrightarrow}
H_*(\bigvee_{d=1}^{\infty}\Po^{dn,m}_n/\Po^{(d-1)n,m}_n,\Z)
$$
is an isomorphism.
Hence, for each $d\geq 1,$ the homomorphism
$$
(\Psi_{d})_*: H_*(\Sigma^{2(mn-2)d}D_d,\Z)\stackrel{\cong}{\longrightarrow}
H_*(\Po^{dn,m}_n/\Po^{(d-1)n,m}_n,\Z)
$$
is an isomorphism.
Therefore $\Psi_d$ is a stable homotopy equivalence.
\end{proof}
%%%(End of proof of Theorem 7.8)%%%

%%%(SECTION 8)%%%
\section{Transfer maps}\label{section: transfer}
%%%%

In this section we prove Lemma \ref{lmm: X}.
For this purpose, we use the transfer maps defined as follows.

%%%%%%%%%%%%%%%%%%%
%%(Definition 8.1)%%
\begin{definition}
%%%%%%%%%%
(i)
%%%%%
For a connected based space $(X,x_0)$,
let $\SP^{\infty}(X)$
denote {\it the infinite symmetric product} of $X$ defined by
$\dis
\SP^{\infty}(X)=\lim_{d\to\infty}\SP^d(X)=\bigcup_{d\geq 0}\SP^d(X),
$
where the space $\SP^d(X)$ is regarded as the subspace of $\SP^{d+1}(X)$ by
identifying 
$\sum_{k=1}^dx_k$ with $\sum_{k=1}^dx_k+x_0$.
So the space $\SP^{\infty}(X)$ may be regarded as the abelian monoid generated by $X$
with the unit $x_0$.
%%%
\par\vspace{1mm}\par
%%%%
(ii)
Since there is a homotopy equivalence
$\Po^{(d-1)n,m}_n\simeq \Po^{dn-1,m}_n$
(by Theorem \ref{crl: stabilization}),
by using this identification wee define the map
$$
\tau :\Po^{dn,m}_n \to \SP^{\infty}(\Po^{dn-1,m}_n)
\simeq \SP^{\infty}(\Po^{(d-1)n,m}_n)
\qquad\mbox{by}
$$
$$
(f_1(z),\cdots ,f_m(z))
\mapsto
\sum_{1\leq i_1,\cdots ,i_m\leq dn}(\prod^{dn}_{k=1, k\not=i_1}(z-a_{k,1}),\cdots ,\prod^{dn}_{k=1,k\not= i_m}(z-a_{k,m})),
$$
where
$(f_1(z),\cdots ,f_m(z))\in \Po^{dn,m}_n$ and
$f_j(z)=\prod^{dn}_{k=1}(z-a_{k,j})$ for $1\leq j\leq m$.
The map $\tau$ naturally extends to a homomorphism of abelian monoid
%%(8.1)%%
\begin{equation}
\tau_{d-1}:\SP^{\infty}(\Po^{dn,m}_n)\to \SP^{\infty}(\Po^{(d-1)n,m}_n).
\end{equation}
For each $1\leq k< d$, define the transfer map
%%(8.2)%%
\begin{equation}
%%%
\tau_{k,d}:\SP^{\infty}(\Po^{dn,m}_n)\to \SP^{\infty}(\Po^{dk,m}_n)
%%%
\end{equation}
%%%
as the composite
$\tau_{k,d}=\tau_k\circ  \tau_{k+1}\circ \cdots \circ \tau_{d-1}$,
i.e.
$$
\tau_{k,d}:
\SP^{\infty}(\Po^{dn,m}_n)
\stackrel{\tau_{d-1}}{\longrightarrow}
\SP^{\infty}(\Po^{(d-1)n,m}_n)
\to \cdots
%\to
%\SP^{\infty}(\Po^{(k+1)n,m}_n)
\stackrel{\tau_{k}}{\longrightarrow}
\SP^{\infty}(\Po^{kn,m}_n).
$$
In particular, %$\tau_{d-1,d}=\tau_{d-1}$ and
 we set $\tau_{d,d}=\mbox{id}$ for $k=d$.
 \qed
%%%%%%%%%%%%%%%%
\end{definition}
%%%%%%%%(End of definition 9.1)%%%
%%

%%(Lemma 8.2)%%
\begin{lemma}\label{lmm: XI}
%%%%%%%%%%%%%
%%(i)%%
$\I$
The induced homomorphism
$(s_{d-1,d})_*:H_*(\Po^{(d-1)n,m}_n,\Z)
\to
H_*(\Po^{dn,m}_n,\Z)$
is a monomorphism.
\par
%%(ii)%%
$\II$
The map
$$
\begin{CD}
\SP^{\infty}(\Po^{dn,m}_n)
@>(\tilde{p}_d,\tau_{d-1,d})>\simeq>
\SP^{\infty}(\Po^{dn,m}_n/\Po^{(d-1)n,m}_n)
\times
\SP^{\infty}(\Po^{(d-1)n,m}_n)
\end{CD}
$$
is a homotopy equivalence,
where
$$
\tilde{p}_k:
\SP^{\infty}(\Po^{kn,m}_n)\to
\SP^{\infty}(\Po^{kn,m}_n/\Po^{(k-1)n,m}_n)
$$
denotes the map
induced from the natural projection
$$\Po^{kn,m}_n\to \Po^{kn,m}_n/\Po^{(k-1)n,m}_n.
$$
%%
%%%
%%%
\end{lemma}
%%%%
%%(Proof of Lemma 8.2)%%
%%
%%%%
%%(Proof of Lemma 8.2)%%
\begin{proof}
%%%%%%%%%%
It is well-known that there is an natural isomorphism 
$\pi_k(\SP^{\infty}(X))\cong \tilde{H}_k(X,\Z)$ for any
connected space $X$ and any $k\geq 0$.
Furthermore,
note that the equality
$\tilde{p}_k\circ \tau_{k,d-1}=
\tilde{p}_k\circ \tau_{k,d}\circ s_{d-1,d}
$
(up to homotopy equivalence)
holds
for each $1\leq k<d$.
Thus we can show that
$$
(\tau_{k,d})_*\circ (s_{d-1,d})_*\equiv
(\tau_{k,d-1})_*
\quad (\mbox{mod Im }(s_{k-1,k})_*)
$$
on $H_*(\Po^{kn,m}_n,\Z)$ for each $1\leq k<d$.
Then by using
\cite[Lemma 2]{Do}, we can prove that 
$(s_{d-1,d})_*$ is a monomorphism and that
the map 
$(\tilde{p}_d,\tau_{d-1,d})$
induces an isomorphism on the homotopy group $\pi_k(\ )$ for any $k$.
Hence, 
the map  $(\tilde{p}_d,\tau_{d-1,d})$ is a homotopy equivalence.
%%%
\end{proof}
%%(End of proof of Lemma 8.2)%%

If we use the above result, it is easy to prove the following result.

%%(Corollary 8.3)%%
\begin{corollary}
%%%%%%%%%
The map
$$
\tilde{\tau}_d=\prod_{k=1}^d\tilde{\tau}_{k,d}:\SP^{\infty}(\Po^{dn,m}_n)
\stackrel{\simeq}{\longrightarrow}
\prod_{k=1}^d\SP^{\infty}(\Po^{kn,m}_n/\Po^{(k-1)n,m}_n)
$$
is a homotopy equivalence,
where 
the map
$\tilde{\tau}_{k,d}$ denotes the composite of maps
$$
\SP^{\infty}(\Po^{dn,m}_n)
\stackrel{\tau_{k,d}}{\longrightarrow}
\SP^{\infty}(\Po^{kn,m}_n)
\stackrel{\tilde{p}_k}{\longrightarrow}
\SP^{\infty}(\Po^{kn,m}_n/\Po^{(k-1)n,m}_n).
\qed
$$
\end{corollary}
%%%%(End of Corollary 8.3)%%

%%(Definition 8.4)%%
\begin{definition}
%%%%%%%%%%%%%%%%%%%%
Let $N$ be a positive integer and assume that $1\leq j<d$.
\par
%%(i)%%
(i)
Let 
$
q_{d.j}^{(N)}:F(\C,d)\times_{S_j\times S_{d-j}}(S^{2N-1})^d
\to 
F(\C,d)\times_{S_d}(S^{2N-1})^d
$
denote the natural covering projection corresponding the subgroup
$S_j\times S_{d-j}\subset S_d$.
Define the transfer 
%%(9.3)%%
\begin{equation}
%%%%%%%%%%%
\sigma^{(N)} :F(\C,d)\times_{S_d}(S^{2N-1})^d
\to
\SP^{\infty}(F(\C,d)\times_{S_j\times S_{d-j}}(S^{2N-1})^d)
\end{equation}
%%%%%%
for the covering projection $q_{d,j}^{(N)}$
by
%%()%%
%\begin{equation*}
$$
\sigma^{(N)} (x)=\sum_{\tilde{x}\in q_{d,j}^{-1}(x)}\tilde{x}
\quad
\mbox{ for }x\in F(\C,d)\times_{S_d}(S^{2N-1})^d.
$$
%\end{equation*}
%%%
\par
%%(ii)%%%
(ii)
Let $\rho_j^{(N)}:F(\C,d)\times_{S_j\times S_{d-j}}(S^{2N-1})^d
\to
F(\C,d)\times_{S_j\times S_{d-j}}(S^{2N-1})^j$
denote the map onto the first $j$ coordinates of $(S^{2N-1})^d$, and
define the map
$$
\sigma_j^{(N)}:
F(\C,d)\times_{S_d}(S^{2N-1})^d
\to
\SP^{\infty}(F(\C,d)\times_{S_j\times S_{d-j}}(S^{2N-1})^j)
$$
by
$\sigma_j^{(N)}=\SP^{\infty}(\rho_j^{(N)})\circ \sigma^{(N)}$.
The map $\sigma_j^{(N)}$ naturally extends to a map
%%(8.4)%%
\begin{equation}
%%%%
\tilde{\sigma}_j^{(N)}:
\SP^{\infty}(F(\C,d)\times_{S_d}(S^{2N-1})^d)
\to
\SP^{\infty}(F(\C,d)\times_{S_j\times S_{d-j}}(S^{2N-1})^j)
\end{equation}
%%%
by the usual addition
$\tilde{\sigma}_j^{(N)}(\sum_kx_k)=\sum_k\sigma_j^{(N)}(x_k)$.
%%%%
\par
%%(iii)%%
(iii)
Let
$$
\mathcal{I}_{j,d}^{\p}:F(\C,d)\times_{S_j\times S_{d-j}}(S^{2mn-3})^j
\to
\Po^{jn,m}_n
$$
denote the $C_2$-structure map given
by the similar way as $\mathcal{I}_d^{\p}$ was defined.
\qed
%%
%%%%
\end{definition}
%%(End of Definition 8.4)%%%%%

%%(Lemma 8.5)%%
\begin{lemma}[\cite{CCMM2}]\label{lmm: sigma e}
%%%%%%%%%%%%%%
Let $1\leq j<d$.
Then the stable map
$$
\sigma_j^{(N)}\circ e_{d,N}:D_d(S^{2N-1})
\to
\SP^{\infty}(F(\C,d)\times_{S_j\times S_{d-j}}
(S^{2N-1})^j)
$$
is null-homotopic.
%%%
\end{lemma}
%%%%%%%
%%(Proof of Lemma 8.5)%%
\begin{proof}
%%%%%%%%%%%%
The case $N=1$ was proved in \cite[page 44-45]{CCMM2} and
the case $N\geq 2$ can be proved completely same way.
%Indeed,
%%%%
\end{proof}
%%%(End of proof of Lemma 8.5)%%%

The following is easy to verify:
%%(Lemma 8.6)%%
\begin{lemma}\label{lmm: CD-transfer}
%%%%%%%%%%%
Let $m$ and $n$ be positive integers $\geq 2$.  
Then if $1\leq j<d$ and $N=mn-1$,
the following diagram is commutative:
%%()%%
\begin{equation*}
%%%
\begin{CD}
\SP^{\infty}(F(\C,d)\times_{S_d}(S^{2mn-3})^d)
@>\tilde{\sigma}_j^{(mn-1)}>>
\SP^{\infty}(F(\C,d)\times_{S_j\times S_{d-j}}(S^{2mn-3})^d))
\\
@V{\SP^{\infty}(\mathcal{I}_d^{\p})}VV 
@V{\SP^{\infty}(\mathcal{I}_{j,d}^{\p})}VV
\\
\SP^{\infty}(\Po^{dn,m}_n) @>{\tau_{j,d}}>>
\SP^{\infty}(\Po^{jn,m}_n)
\quad
\qed
\end{CD}
\end{equation*}
%%%
\end{lemma}
%%%%%%%(End of Lemma 8.6)%%

%%(Lemma 8.7)%%
\begin{lemma}\label{lmm: null homotopic}
%%%%%%%%%%%%%%
If $m\geq 2$ and $n\geq 2$ be positive integers,
the stable map
$$
\tau_{d-1,d}\circ \SP^{\infty}(\mathcal{I}_d^{\p})\circ
\SP^{\infty}(e_d):
\SP^{\infty}(\Sigma^{2(mn-2)d}D_d) \to
\SP^{\infty}(\Po^{(d-1)n,m}_n)
$$
is null-homotopic.
%%%
\end{lemma}
%%%%%%
%%(Proof of Lemma 8.7)%%
\begin{proof}
%%%%%
Note that
$\sigma_{d-1}^{(mn-1)}\circ e_d$ is null-homotopic 
by  
Lemma \ref{lmm: sigma e} (cf.  (\ref{eq: ed def})).
By  Lemma \ref{lmm: CD-transfer} we see that
\begin{eqnarray*}
%%
%\tau_{d-1}\circ \SP^{\infty}(\mathcal{I}_d^{\p})\circ
%\SP^{\infty}(e_d)
%&=&
\tau_{d-1,d}\circ \SP^{\infty}(\mathcal{I}_d^{\p})\circ
\SP^{\infty}(e_d)
&= &
\SP^{\infty}(\mathcal{I}_{d-1,d}^{\p})\circ
\tilde{\sigma}_{d-1}^{(mn-1)}\circ \SP^{\infty}(e_d)
\\
&=&
\SP^{\infty}(\mathcal{I}_{d-1,d}^{\p})\circ
\SP^{\infty}(\sigma_{d-1}^{(mn-1)}\circ e_d)
\simeq *.
%%%
\end{eqnarray*}
%%%%
Thus
the map
$\tau_{d-1,d}\circ \SP^{\infty}(\mathcal{I}_d^{\p})\circ
\SP^{\infty}(e_d)$ is null-homotopic.
\end{proof}
%%(End of Proof of Lemma 8.7)%%

Now it is ready to prove Lemma \ref{lmm: X}.

%%(Proof of Lemma 7.3)%%
\begin{proof}[Proof of Lemma \ref{lmm: X}]
%%%%%%%%%%%%%
Since the first assertion follows from
(i) of Lemma \ref{lmm: XI},  it suffices to prove 
the assertion (ii).
First,
it follows from Theorem \ref{thm: VIII} that the homomorphism
$
(\mathcal{I}_d^{\p}\circ e_d)_*:
H_*(\Sigma^{2(mn-2)d}D_d,\Z) \to H_*(\Po^{dn,m}_n,\Z)
$
is a monomorphism.
Next,
by (ii) of Lemma \ref{lmm: XI} the homomorphism
$$
\begin{CD}
%%%%%
H_*(\Po^{dn,m}_n)
@>(\tilde{p}_{d*},\tau_{d-1,d*})>\cong>
H_*(\Po^{dn,m}_n/\Po^{(d-1)n,m}_n)\oplus
H_*(\Po^{(d-1)n,m}_n)
\end{CD}
$$
is an isomorphism.
Hence, the composite
$((\tilde{p}_{d})_*\circ (\mathcal{I}_d^{\p}\circ e_d)_*,
(\tau_{d-1,d})_*\circ (\mathcal{I}_d^{\p}\circ e_d)_*)$
$$
H_*(\Sigma^{2(mn-2)d}D_d,\Z)
\to
H_*(\Po^{dn,m}_n/\Po^{(d-1)n,m}_n,\Z)\oplus
H_*(\Po^{(d-1)n,m}_n,\Z)
$$
is a monomorphism.
However, by Lemma \ref{lmm: null homotopic},
$(\tau_{d-1,d})_*\circ (\mathcal{I}_d^{\p}\circ e_d)_*$ is trivial.
Thus, the homomorphism
$$
(\tilde{p}_d)_*\circ (\mathcal{I}_d^{\p}\circ e_d)_*: 
H_*(\Sigma^{2(mn-2)d}D_d,\Z)
\to
H_*(\Po^{dn,m}_n/\Po^{(d-1)n,m}_n,\Z)
$$
is a monomorphism.
Since
$\Psi_d=\tilde{p}_d\circ \mathcal{I}_d^{\p}\circ e_d$,
the homomorphism
$$
(\Psi_d)_*: 
H_*(\Sigma^{2(mn-2)d}D_d,\Z)
\to
H_*(\Po^{dn,m}_n/\Po^{(d-1)n,m}_n,\Z)
$$
is a monomorphism.
%%%
\end{proof}
%%(End of proof of Lemma 8.3)%%%%

%\par\vspace{2mm}\par
\noindent{\bf Funding. }
The second author was supported by 
JSPS KAKENHI Grant Number 26400083, Japan.
%%%%%%%%%%%%%%%%%%%%
%\par\vspace{2mm}\par
%\noindent{\bf Acknowledgements. }
%%%%%%%%%
%The authors should like to take this opportunity to thank  Masahiro Ohno 
%and Farb Benson
%for his many valuable  insights and suggestions concerning resultants.
%The second author was supported by 
%JSPS KAKENHI Grant Number 26400083, Japan.
%%%%
%    Bibliographies can be prepared with BibTeX using amsplain,
%    amsalpha, or (for "historical" overviews) natbib style.
\bibliographystyle{amsplain}
%    Insert the bibliography data here.

%%%%%%%%%%%%%%%%%%%%%%%%%%%%%%%%%%%%%%%%%%%%%%%%%%

\end{document}